\documentclass[a4paper, 12pt]{article}
\usepackage[right=3cm,left=3cm,top=3cm,bottom=3cm]{geometry}

\usepackage{amsmath,amsfonts,amsthm,amssymb}
\usepackage{mathtools}

\usepackage{enumitem}
\setlist[itemize]{leftmargin=4mm}

\usepackage{color}

\usepackage{hyperref}

\usepackage[utf8]{inputenc}
\usepackage[T1]{fontenc}

\newcommand{\p}{\partial}

\newcommand{\Q}{{\mathbb Q}}
\newcommand{\C}{{\mathbb C}}
\newcommand{\R}{{\mathbb R}}
\newcommand{\Z}{{\mathbb Z}}
\newcommand{\N}{{\mathbb N}}

\newcommand{\T}{{\mathbb T}}

\newcommand{\sS}{{\cal S}}
\newcommand{\TT}{{\cal T}}

\newcommand{\lag}{\langle}
\newcommand{\rag}{\rangle}
\newtheorem{theorem}{Theorem}[section]
\newtheorem{lemma}[theorem]{Lemma}
\newtheorem{proposition}[theorem]{Proposition}
\newtheorem{corollary}[theorem]{Corollary}
\theoremstyle{definition}

\theoremstyle{remark}
\newtheorem{remark}[theorem]{Remark}

\numberwithin{equation}{section}

\title{$H^1$ local exact controllability of some one-dimensional bilinear Schrödinger~equations}
\author
{
{Nabile Boussa\"id\,\footnote{Universit\'{e} Marie et Louis Pasteur, CNRS, Institut UTINAM, \'{E}quipe de physique th\'{e}orique,F-25000 Besan\c{c}on, France}}
\and
{Alessandro Duca\,\footnote{Universit\'e de Lorraine, CNRS, Inria, IECL, F-54000 Nancy, France}}
}
\date{}

\begin{document}

\maketitle

\begin{abstract}
The local exact controllability of the one-dimensional bilinear Schrödinger equation with Dirichlet boundary conditions has been extensively studied in subspaces of $H^3$ since the seminal work of K. Beauchard. Our first objective is to revisit this result and establish the controllability in $H^1_0$ for suitable discontinuous control potentials. In the second part, we consider the equation in the presence of periodic boundary conditions and a constant magnetic field. We prove the local exact controllability of periodic $H^1$-states, thanks to a Zeeman-type effect induced by the magnetic field which decouples the resonant spectrum. Finally, we discuss open problems and partial results for the Neumann case and the harmonic oscillator.


\medskip

\end{abstract}

\section{Introduction}
Bilinear quantum control theory investigates the possibility of driving quantum systems to desired states through the application of external forces.
For instance, in studying the dynamics of a charged particle, 
a viable control mechanism is the application of external electromagnetic potential fields, which influence and steer its evolution. 
Mathematically, the quantum state associated with the particle moving on a domain $\Omega\subset \R^d$, $d\in\N$, can be represented as a unitary state $\psi\in L^2(\Omega,\C)$ and its evolution is governed by the Schrödinger equation
 \begin{equation}
 	\label{0.1}
 	i \p_t \psi =\Big[-\big(div+i A(t,x) \big)\circ \big(\nabla +i  A(t,x)\big)+ E(t,x)\Big]\, \psi,\, \quad x\in \Omega,\ t>0,
 \end{equation}
 in the presence of suitable boundary conditions. In this formulation, the magnetic potential is represented as $A(t,x)$, while $E(t,x)$ models the electric potential.

In practice, it is often useful to employ external fields with controllable intensity and predefined spatial distributions. These types of controls are usually referred to as ``bilinear'', and an example is \begin{align}\label{separate}A(t,x)=u_0(t)\mu_0(x),\qquad \text{ and }\qquad \ E(t,x)=u_1(t)\mu_1(x) .\end{align}
 Here, the functions $u_0:\R^+\rightarrow\R$ and $u_1:\R^+\rightarrow\R$ describe the time-dependent intensity of the magnetic and electric fields, respectively. The spatial distributions of these potentials are captured by the functions $\mu_0:\Omega\rightarrow\R^d$ for the magnetic field and $\mu_1:\Omega\rightarrow\R$ for the electric field.

In this work, we investigate the local exact controllability of the Schrödinger equation~\eqref{0.1} in the one-dimensional case $\Omega = (0,1)$ via bilinear controls as~\eqref{separate}. We start with the case where Dirichlet boundary conditions are satisfied. This problem has been extensively studied over the past two decades,
starting with the foundational work by Beauchard~\cite{Bea05}, later improved by Beauchard and Laurent~\cite{BL-2010} who considered suitable subspaces of $H^3$ and a neighborhood of the ground state.
We revisit the existing theory for the Dirichlet case: we achieve the local exact controllability within the space $H^1_0$ nearby some eigensolutions. Then, we consider the case of periodic boundary conditions in the presence of Zeeman-type effect. We ensure the local exact controllability for periodic $H^1$-functions nearby any eigensolution. Finally, we discuss some open problems concerning the Neumann boundary conditions and the quantum harmonic oscillator. 
For instance, in the case of the harmonoc oscillator, we prove the global exact controllability result of a suitable linearised system.

\subsection*{The local exact controllability in the Dirichlet Case}
Let us start by considering the Schrödinger equation~\eqref{0.1} with $\Omega=(0,1)$ and Dirichlet-type boundary conditions. 
We set $A=0$ and $E(t,x)=u(t)\mu(x)$ which allows us to rewrite~\eqref{0.1} in the form
 \begin{equation}
 \label{0.1_d}\begin{cases}
 	i \p_t \psi =-\Delta\psi+ u(t)\mu(x)\psi, & \quad x\in (0,1),\ t>0,\\
  \psi(t,x=0)=\psi(t,x=1)=0,& \quad t>0,\\
  \psi_0(t=0)=\psi_0\in L^2((0,1),\C).
\end{cases}
\end{equation}
The problem is considered in the Hilbert space $L^2=L^2((0,1),\C)$ endowed with the norm $\|\cdot\|_{L^2}$ associated with the hermitian product
$$\lag f,g \rag_{L^2}=\int_0^{1} f(x) \overline{g}(x) dx,\qquad  \forall f,g\in L^2((0,1),\C).$$
The Dirichlet Laplacian appearing in~\eqref{0.1_d} is self-adjoint and with compact resolvent. Its spectrum is composed by isolated and simple eigenvalues $(\lambda_k)_{k\in\N^*}$ such that
$$\lambda_k={k^2\pi^2},$$
and a corresponding to a family of eigenfunctions, forming a Hilbert basis for $L^2((0,1),\C)$, is such that
$$\phi_k(x)= \sqrt{2}\sin\big({k}\pi x\big),\qquad  x\in (0,1),\quad \forall k\in\N^*.$$
 We also introduce the Hilbert space $H^1_0=H^1_0((0,1),\C)$ endowed with the norm~
$$\|\cdot\|_{H^1_0}=\sqrt{\sum_{k=1}^\infty|k\lag \phi_k,\cdot\rag_{L^2}|^2},$$
which is equivalent to the common norm $\|\partial_x\cdot\|$. 
Let $A:=\{a_j\}_{j\leq N}\subset (0,1)$ with $N\in\N^*$, and the space
\begin{align*}
H^1(A)&=\Big\{\psi \in L^2((0,1),\R)\,:\,\psi \text{ is of class $H^1$ up to the discontinuities } \{a_j\}_{j\leq N}\Big\}\\
&=H^1((0,a_1),\R)\times\Big(\prod_{j=1}^{N-1}H^1((a_j,a_{j+1}),\R)\Big)\times H^1((a_N,1),\R)
\end{align*}
endowed with the norm
$$\|\cdot\|_{H^1(A)}:=\sqrt{\|\cdot\|_{H^1(0,a_1)}^2+...+\|\cdot\|_{H^1(a_N,1)}^2}.$$
We introduce the eigensolutions of the Dirichlet Laplacians
$$\phi_l(t) =e^{-i\lambda_l t}\phi_l,\qquad  \forall l\in\N^*.$$
We are now ready to present the well-posedness of~\eqref{0.1_d} and its local exact controllability in $H^1_0$.

\begin{theorem}\label{theorema}
Let $A:=\{a_j\}_{j\leq N}\subset (0,1)$ with $N\in\N^*$ and $\mu\in H^1(A)$.

\smallskip

\noindent
{\bf 1.} Let $T>0$, $\psi_0\in H^1_0$ and $u\in L^2( (0,T) ,\R)$. Then, there
exists
a unique
mild
solution $\psi \in C( [0,T] ,H^1_0)$
to the problem~\eqref{0.1_d}
that is a solution to
the Duhamel formula:
$$
\psi(t)=e^{i\Delta t} \psi_0-i\int_0^te^{i\Delta(t-s)} u(s)\mu \psi(s) d s, \quad t\in [0,T] .$$

\smallskip

\noindent
{\bf 2.} Let $l\in\N^*$ such that
\begin{align}\label{resonant} j^2 - l^2\neq l^2-k^2,\qquad  \forall j,k\in\N^*\setminus\{l\}.\end{align}
Assume there exists a constant $C>0$ such that
\begin{equation}\label{1.9_d}
|\lag \mu\phi_l , \phi_{k}\rag_{L^2}|	\ge \frac{C}{k},\qquad \qquad  \forall k\in\N^*.
\end{equation} 
For any $T>0$, there exists $\delta>0$ such that, for~any~$\psi_0, \psi_1 \in H^1_0$ with $\|\psi_0\|_{L^2}=\|\psi_1\|_{L^2}=1$ and
\begin{equation*}
 \|\psi_0-\phi_l\|_{H^1_0}<\delta, \qquad \qquad \qquad  \|\psi_1-\phi_l(T)\|_{H^1_0}<\delta,
\end{equation*}
 there exists a control $u\in L^2((0,T),\R)$ such that the corresponding solution $\psi$ of  the problem~\eqref{0.1_d} with $\psi(0)=\psi_0$ satisfies
 $$\psi(T)= \psi_1.$$
 \end{theorem}

\begin{proof}
The first point of Theorem~\ref{theorema} follows from Theorem~\ref{well_d} while the second is Theorem~\ref{local-exact-d}.
\end{proof}

The first main novelty of Theorem~\ref{theorema} is the well-posedness of~\eqref{0.1_d} in $H^1_0$ even though the operator $M_\mu: \psi \mapsto \mu\psi$ is not continuous in $H^1_0$. The second is the local exact controllability in $H^1_0$ which is larger than the $D(|\Delta|^\frac{3}{2})\subset H^3$ usually considered in this type of results (see for instance~\cite{BL-2010}).

The control result of Theorem~\ref{theorema} holds in the vicinity of an eigensolutions $\phi_l(t)$ with $l\in\N^*$ verifying~\eqref{resonant}: examples of $l$ verifying the condition are
$l\in\{1,\ 2,\ 3,\ 4\}$,
while a mode which is not included in the result is for instance $\phi_5(t)$ since $7^2-5^2=5^2-1^2$ (and then $\phi_{5 n}(t)$ with $n\in\N^*$). Note that $6$, $7$, $8$ or $9$ are other suitable examples for the controllability.
Now, an example of a control potential $\mu$ verifying Theorem~\ref{theorema} when $l=1$ is
$$\mu= {\mathbf 1}_{[0,1/2]}+{\mathbf 1}_{[1/4,3/4]}.$$
Indeed, the estimate~\eqref{1.9_d} is ensured by the relations, for every $k\in\N^*$:
$$|\lag \phi_{2k},\mu\phi_1\rag_{L^2}|= \Big|\frac{2 (-1)^k k}{-\pi + 4 k^2 \pi} \Big|
, $$
$$|\lag \phi_{2k-1},\mu\phi_1\rag_{L^2}|=\Big|\frac{ k \cos(k \pi /2) + (1 - k) \sin(k \pi /2))}{2 k (k-1) \pi}\Big|
.$$

\subsection*{The local exact controllability in the periodic case}
Let us now consider the Schrödinger equation~\eqref{0.1} with $\Omega=(0,1)$ in the case of periodic boundary conditions, which corresponds to the framework of the one-dimensional torus $\T$. We consider the specific case of
$$A(t,x)=-\frac{u_0}{2},\qquad \qquad E(t,x)=u_1(t) \mu(x)-\frac{|u_0|^2}{4},$$
where $u_0\in\R$ and $\mu_1:(0,1)\rightarrow \R$ are fixed and $u_1:\R^+\rightarrow \R$ is the actual control. Under these hypotheses, the Schrödinger equation~\eqref{0.1} reads
 \begin{equation}
 \label{0.1_p}
 \begin{cases}
 	i \p_t \psi =-\Delta\psi+ u_0 P\psi +u_1(t)\mu(x)\psi, & \quad x\in (0,1),\ t>0,\\
  \psi(t,x=0)=\psi(t,x=1),& \quad t>0,\\
  \p_x\psi(t,x=0)=\p_x\psi(t,x=1),& \quad t>0,\\
  \psi_0(t=0)=\psi_0\in L^2((0,1),\C),
\end{cases}
\end{equation}
where $P$ is the momentum operator given by $$P=i\partial_x.$$ We define then a family of eigenfunctions of the Laplacian $\Delta$ under periodic boundary conditions, which forms a Hilbert basis for $L^2((0,1),\C)$, as
$$\phi_k(x)= e^{i 2 k \pi x},\qquad  \forall k\in\Z.$$
The functions $\{\phi_k\}_{k\in\Z}$ are also eigenmodes of the operator $H=-\Delta+ u_0 P$ and they correspond to some eigenvalues $(\lambda_k)_{k\in\Z}$ such that
$$\lambda_k=4 \pi^2k^2 - 2 \pi u_0 k,\qquad  \forall k\in\Z.$$
We consider the Hilbert spaces $H^1_p$ as the space of $H^1$ functions which are periodic, {\it i.e.}
$$H^1_p=\big\{\psi\in H^1((0,1),\C)\,:\, \psi( 0)=\psi( 1)\big\}$$ 
endowed with the norm
$$\|\cdot\|_{H^1_p}=\sqrt{|\lag \phi_0,\cdot\rag_{L^2}|^2+\sum_{k\in\Z^*}^\infty|k \lag \phi_k,\cdot\rag_{L^2}|^2},$$
which is equivalent to the norm $\|\p_x \cdot\|_{L^2}$. Note that $H^1_p$ can be identified with $H^1(\T,\C)$. We introduce the eigensolutions of the periodic Laplacian
$$\phi_l(t) =e^{-i\lambda_l t}\phi_l,\qquad  \forall l\in\Z.$$
We now state our well-posedness and the local exact controllability result for the Schrödinger equation~\eqref{0.1_p}.

\begin{theorem}\label{theoremb}
Let $A:=\{a_j\}_{j\leq N}\subset (0,1)$ with $N\in\N^*$ and $\mu\in H^1(A)$. 

\smallskip

\noindent
{\bf 1.} Let $T>0$, $\psi_0\in H^1_p$, $u_0\in\R$ and $u_1\in L^2( (0,T) ,\R)$. Then, there is a unique mild solution $\psi \in C( [0,T] ,H^1_p)$ to the problem~\eqref{0.1_p} defined by the Duhamel's formula:
$$\psi(t)=e^{-i(-\Delta+u_0P) t} \psi_0-i\int_0^te^{-i(-\Delta+u_0P)(t-s)} u_1(s)\mu \psi(s) d s, \quad t\in [0,T] .$$

\smallskip

\noindent
{\bf 2.} Let $u_0\in \R\setminus 2\pi \Q$ and $l\in\Z$. Assume there exists a constant $C>0$ such that
\begin{equation}\label{1.9}
|\lag \mu \phi_l, \phi_{k}\rag_{L^2}|	\ge \frac{C}{|k|+1},\qquad \qquad  \forall k\in\Z.
\end{equation}
For any $T>0$, there exists $\delta>0$ such that, for~any~$\psi_0, \psi_1 \in H^1_p$ with $\|\psi_0\|_{L^2}=\|\psi_1\|_{L^2}=1$ and
\begin{equation*}
 \|\psi_0-\phi_l\|_{H^1_p}<\delta, \qquad \qquad \qquad  \|\psi_1-\phi_l(T)\|_{H^1_p}<\delta,
\end{equation*}
there exists a control $u_1\in L^2((0,T),\R)$ such that the corresponding solution $\psi$ of  the problem~\eqref{0.1_p} with $\psi(0)=\psi_0$ satisfies  $$\psi(T)= \psi_1.$$
\end{theorem}

\begin{proof}
The first point of Theorem~\ref{theoremb} follows from Theorem~\ref{well_p}, while the second 
is Theorem~\ref{local-exact-p}.
\end{proof}

As for Theorem~\ref{theorema}, there are two main novelties of Theorem~\ref{theoremb}. First, the well-posedness of~\eqref{0.1_p} holds in $H^1_p$ despite the lack of regularity of the function $\mu$. Second, we demonstrate local exact controllability within the same space. Unlike the Dirichlet case, here we are able to control neighborhoods of any eigensolution $\phi_l(t)$ with $l \in \Z$, thanks to the presence of the Zeeman-type effect.

An example of control potential $\mu$ verifying Theorem~\ref{theorema} nearby $\phi_0(t)\equiv 1$ is
$$\mu=x {\mathbf 1}_{[0,1/2]}.$$
Indeed, the identity~\eqref{1.9} follows from the computations, for every $k\in\Z^*$,
$$|\lag \phi_k,\mu\rag_{L^2}|=\Big|\int_0^{1/2} x e^{2 i k \pi x}d x \Big|= \Big|\frac{(-1)^{k+1} + 1 - i k \pi}{4 k^2 \pi^2}\Big|.$$


\subsection*{The key aspects of the proof}
Let us consider the Dirichlet case.
The approach leading to Theorem~\ref{theorema} is inspired by the strategy in~\cite{BL-2010} and consists of the following steps:
\begin{enumerate}
\item We prove the well-posedness in $D(|\Delta|^\frac{n}{2})$ with a suitable $n\in\N^*$, even though $M_\mu:\psi\mapsto \mu\psi$ is not continuous in such a space. 
This prevents the obstruction to the exact controllability of bilinear equations
in~\cite{Ball}, 
which arises in any Banach space $X $ preserved by the dynamics~\eqref{0.1_d} 
whenever $M_\mu$ is a bounded operator in $X$. 
The proof exploits crucially a hidden regularising effect of the dynamics adapted from~\cite{BL-2010}.

\item We prove the local exact controllability in the space $D(|\Delta|^\frac{n}{2})$ when
\begin{equation}\label{Eq:SpectMu}
\lag \mu\phi_l , \phi_{k}\rag_{L^2}\geq C k^{-n},\quad k\in \N^*,
\end{equation}
for some $C>0$.
The result is deduced via the Inverse Mapping Theorem and the global exact controllability of a suitable linearised system. The latter is deduced from the solvability of a moment problem of the form
\begin{align}\label{mome}
x_k = \int_0^{T} e^{i(\lambda_k-\lambda_l)s }u(s)ds,\qquad \ x_k = \frac{i e^{i\lambda_{k} T}\lag \psi,\phi_{k}\rag_{L^2}}{\lag \mu\phi_l ,\phi_k\rag_{L^2}},\ k\in\N^*,
\end{align}
where $\psi$ 
belongs to the controlled space. 
When $\psi \in D(|\Delta|^\frac{n}{2})$ and~\eqref{Eq:SpectMu} is verified, the sequence $(x_k)_{k\in\N^*} \in \ell^2$, and 
an Ingham-type theorem infers the solvability of~\eqref{mome}.
\end{enumerate} 

The matching of the ``spectral'' property~\eqref{Eq:SpectMu} of $\mu$ and the regularity of $D(|\Delta|^\frac{n}{2})$, the space where the well-posedness is ensured, is the core of the strategy leading to local exact controllability.
In~\cite{BL-2010}, for example, the Schrödinger equation~\eqref{0.1_d} is considered with $\mu \in H^3$ and the lack of regularity of the function $\mu$ at the boundaries $x = 0$ and $x = 1$ is tempered by the Dirichlet boundary conditions. Indeed, $M_\mu$ remains continuous in $D(\Delta) = H^2 \cap H^1_0$, and continuity is only lost in higher regularity spaces as $D(|\Delta|^\frac{3}{2})$.
For this reason, in~\cite{BL-2010}, the well-posedness and subsequently the exact controllability 
are proved in $D(|\Delta|^\frac{3}{2})$.
In Theorem~\ref{theorema}, the lack of regularity for the potential $\mu$ arises at interior points due to the choice $\mu \in H^1(A)$. This choice prevents the regularising effect of Dirichlet boundary conditions, and as a result, the operator $M_\mu$ is not continuous on $H^1_0$, which allows us to control within this space.

The local exact controllability in the periodic case is still an open question. A possible reason may be that a bilinear Schrödinger equation
\begin{align}\label{intro_eq}
i\p_t\varphi=-\Delta \varphi +u(t)\mu(x)\varphi, \quad x\in (0,1),\ t>0, \end{align}
cannot be directly addressed with the above strategy when periodic boundary conditions are considered. Indeed, the Laplacian operator has double eigenvalues in this case, and the solvability of the moment problem cannot be directly deduced as in the Dirichlet case.  Theorem~\ref{theoremb} exploits the presence of the momentum operator in~\eqref{0.1_p}, which corresponds to the use of the magnetic field when the problem is viewed from the perspective of~\eqref{0.1}. Specifically, the presence of $P$ in~\eqref{0.1_p} leads to consider the principal operator $H = -\Delta + u_0 P$ and to exploit a Zeeman-type effect on the resonant spectrum. With a suitable choice of $u_0$, the operator $H$ has a simple spectrum, enabling us to apply the strategy introduced for the Dirichlet framework.

It is natural to ask if these techniques can also be applied to the Schrödinger equation~\eqref{intro_eq} in the presence of Neumann boundaries or 
for the quantum harmonic oscillator in the presence of a bilinear control. In the latter case, the equation in $L^2(\mathbb{R}, \mathbb{C})$ to consider is given by:
$$ i\p_t \psi = [-\Delta + x^2] \psi + u(t) \mu(x) \psi, \quad x \in \R, \, t > 0. $$
The proof of well-posedness applies to both problems, respectively in $H^1((0,1), \C)$ and in a suitable subspace of $H^1(\R, \C)$, as detailed in Propositions~\ref{well_n} and~\ref{well_h}. 
Even though the proposed strategy to get local controllability may seem valid at the abstract framework, we are unable to provide explicit examples of control fields satisfying the required hypotheses. 
Further discussion on these aspects can be found in Section~\ref{comments_n} and~\ref{comments_h}. 
Interestingly, the obtructions in the perdioc and Neuman case are related, see Remark~\ref{periodic} below.
In the case of the harmonic oscillator, we can prove a global exact controllability result for a suitable linearised system, see Theorem~\ref{theoremh} below.

 \subsection*{Related bibliography}

The exact controllability of the Schrödinger equation, with bilinear control, is a delicate matter due to the obstruction to the exact controllability proved by Ball, Marsden, and Slemrod~\cite{Ball}. In~\cite{Ball}, they consider bilinear equations of the form $\p_t\varphi = A \varphi + u(t)B\varphi$ in a Banach space $X$, when $A$ generates a $C^0$-semigroup in $X$ and $B$ is a bounded operator in $X$. They show that the reachable set of its dynamics, starting from any state in $X$ and when $u \in L^2_{\text{loc}}$, is contained in a countable union of compact subsets of $X$. When $X$ is infinite-dimensional, this property implies that the reachable set has a dense complement, and therefore we cannot expect exact controllability results in $X$. This property translates to the bilinear Schrödinger equation to the unit sphere of $L^2$, see Turinici~\cite{Turinici}, and to higher regularity spaces or with unbounded operators $B$, see Boussa\"{\i}d, Caponigro and Chambrion~\cite{UP}.

For this reason, it has been necessary to exploit different notions of controllability, such as approximate controllability or exact controllability in higher regularity subspaces of $L^2$ that are not preserved by the control operator.

The global approximate controllability of the bilinear Schrödinger equation has been widely studied in the literature. Many results were proved in the last decades for the controllability in large time. On this subject, we refer to~\cite{Mir09, Ner10} for Lyapunov techniques, while we cite~\cite{BCMS12, BGRS15} for adiabatic arguments and~\cite{BdCC13,BCMS09} for Lie-Galerkin methods. The global approximate controllability in small time has instead been an open problem until the work~\cite{Vahagn} very recently. In this paper, the authors studied the bilinear Schrödinger equation, proving the approximate controllability in small time between the eigenmodes. Initiated by this work, several other results have been provided on small-time approximate controllability; on this subject, we refer to~\cite{ BP04-1,BP04-2, small-time-molecule, coron-2023,small-time-momentum, small-time-wave}.

The turning point for the study of the exact controllability of the bilinear Schrödinger equation was the idea of controlling the equation in suitable subspaces of $D(\Delta)$, introduced by Beauchard in~\cite{Bea05} and comprehensively developed by Beauchard and Laurent in~\cite{BL-2010} with the emphasis of the so-called hidden regularity. Following this approach, several works have been developed to push forward the theory. For local exact controllability results, we refer to~\cite{Duc19}, while we mention~\cite{Duc20b,Mor14, MN15} for simultaneous exact controllability results. The bilinear controllability in the presence of nonlinearities has been studied in~\cite{Karine2, Vahagn2}, and the framework of network-type domains has been explored in~\cite{Duc20,Duc21}. We also refer to the very recent work~\cite{Karine3}, which presents an obstruction to exact controllability in the domain of the Neumann Laplacian in the case of Neumann boundary conditions.

Due to the obstruction from~\cite{Ball}, as emphasised in~\cite{UP}, a key feature of the potential is that it should be unbounded as an operator from the domain of the Laplacian to the domain itself. When a boundary condition is present, it is enough to consider a multiplication operator which does not preserve such a boundary condition (but necessarily at some higher-order derivative). We propose here to consider a multiplication operator which does not preserve the regularity (this allows low-regularity subspaces).    

\subsection*{Scheme of the work}
The paper is organised as follows. In Section~\ref{Dirichlet}, we consider the bilinear Schrödinger equation~\eqref{0.1_d} and prove Theorem~\ref{theorema}. We first show the well-posedness and then prove the local exact controllability. Section~\ref{section_p} is dedicated to~\eqref{0.1_p} and the proof of Theorem~\ref{theoremb}. In the first part of the section, we provide some spectral properties for the problem and establish the well-posedness; the second part is dedicated to the exact controllability result. In Section~\ref{comments_n}, we discuss the well-posedness and some possible open problems involving the bilinear Schrödinger equation with Neumann boundary conditions. In Section~\ref{comments_h}, we discuss the harmonic oscillator and present a result on the global exact controllability for a suitable linearised system. In Appendix~\ref{endpoint}, we prove the $C^1$-regularity for the endpoint maps. Appendix~\ref{mome_appendix} contains two well-known results of solvability of moment problems used in different parts of the present work.

\subsection*{Acknowledgements}
The authors thank Thomas Chambrion, Vahagn Nersesyan, Tak\'eo Takahashi and Laurent Thomann for the fruitful discussions. 

This work has been supported by the EIPHI Graduate School (contract ANR-17-EURE-0002) and by the Bourgogne-Franche-Comté Region and funded in whole or in part by the French National Research Agency (ANR) as part of the QuBiCCS project ``ANR-24-CE40-3008-01''.

\section{The Dirichlet case}\label{Dirichlet}
This section is dedicated to the analysis of the bilinear Schrödinger equation~\eqref{0.1_d}. We start by studying the well-posedness of the equation in $H^1_0$ when $\mu$ is piecewise regular and, in particular, belongs to $H^1(A)$ with $A = \{a_j\}_{j \leq N} \subset (0,1)$ and $N \in \N^*$. Afterwards, we prove the local exact controllability in this framework, and we provide an explicit example of a potential satysfying all the assumptions. The results of this section constitute the proof of Theorem~\ref{theorema}.

\subsection{Well-posedness}
Consider the following linear Schrödinger  equation with a source term:
\begin{equation}
\label{0.1_d_bis}\begin{cases}
 	i \p_t \psi =-\Delta\psi+ u(t)\mu(x)\psi + f(t,x) & \quad x\in (0,1),\ t>0,\\
  \psi(t,x=0)=\psi(t,x=1)=0,& \quad t>0,\\
  \psi_0(t=0)=\psi_0\in L^2((0,1),\C).
\end{cases}
\end{equation}
Recall that $\{\phi_k\}_{k\in\N^*}$, 
$$\phi_k(x)= \sqrt{2}\sin\big({k}\pi x\big),\qquad  x\in (0,1),\, \forall k\in\N^*,$$
is a family of eigenfunctions of the Dirichlet Laplacian forming a Hilbert basis for $L^2$.

We now state the well-posedness of the Schrödinger equation~\eqref{0.1_d_bis} which genralises~\eqref{0.1_d}.
\begin{theorem}\label{well_d}
Let
$N \in \N^*$,
$A = \{a_j\}_{j \leq N} \subset (0,1)$ and
$T>0$.
Let $\mu\in H^1(A)$, $\psi_0\in H^{1}_0$ and $f \in L^2( (0,T) , H^{1}(A))$.
For any $u\in L^2( (0,T) ,\R)$, there exists a unique mild solution $\psi \in C( [0,T] ,H^{1}_0)$ to the problem~\eqref{0.1_d_bis} such that $\psi(0)=\psi_0$ that is a solution to
$$
\psi(t)=e^{i\Delta t} \psi_0-i\int_0^te^{i\Delta(t-s)}(u(s)\mu \psi(s) +  f(s))\,\mathrm{d} s, \quad t\in [0,T].
$$
Moreover, there is a constant~$C_T>0$ such that
$$\|\psi\|_{C( [0,T], H^{1}_0 )}\le C_T \left (\|\psi_0\|_{H^1_0}+\|f\|_{L^2( (0,T) , H^{1}(A))}\right).$$
\end{theorem}

The proof of Theorem~\ref{well_d} is based on the techniques adopted in the proof of~\cite[Proposition 2]{BL-2010}, and the central part is the following lemma.

\begin{lemma}\label{wellposedness-bound-3.d}
Let
$N \in \N^*$,
$A = \{a_j\}_{j \leq N} \subset (0,1)$,
and $f\in L^2( (0,T) , H^{1}(A))$.
The function $G:t\mapsto \int_0^te^{- i\Delta s} f(s) d s$ belongs to $C( [0,T] ,H^1_0)$ and there is a constant $C_T > 0$ such that
\begin{equation*}
\|G \|_{C( [0,T] ,H^1_0)}\le C_T \|f\|_{L^2( (0,T) , H^{1}(A))}.
\end{equation*}
\end{lemma}
\begin{proof}
First, since $e^{-i\Delta s}$ is unitary on $L^2((0,1),\C)$
and $f\in L^2( (0,T) , L^2((0,1),\C))$, $G$ belongs to $C( [0,T] ,L^2((0,1),\C))$.
To prove $G\in C( [0,T], H^1_0 )$ let us use a spectral decomposition. Note that $$G(t)=\sum_{k\in\N^*}\phi_k \int_0^t e^{i\lambda_k s}\lag \phi_k,f(s)\rag_{L^2} d s.$$
Recall $\phi_k=\sqrt{2}\sin(k \pi \cdot)$ are the eigenfunctions of the Dirichlet Laplacian.
Let
$$g:=\big(\p_xf\big)1_{(0,a_1)}+...+\big(\p_xf\big)1_{(a_N,1)} \in L^2\Big( (0,T) , L^2\big((0,1),\C\big)\Big).$$
Now, let us set $a_0:=0$ and $a_{N+1}:=1$. Then
 \begin{align}
 \frac{k\pi }{\sqrt{2}}\lag \phi_k,f (\tau,\cdot)\rag_{L^2}
&= -\lag \p_x\cos(k \pi x),f (\tau,\cdot)\rag_{L^2}\nonumber\\
&= \lag \cos(k \pi x), g(\tau,\cdot)\rag_{L^2}
\nonumber\\
&\,
-\sum_{j=1}^{N+1} \cos(k \pi a_j^-) f (\tau,a_j^-)-\sum_{j=0}^N \cos(k \pi a_j^+)f (\tau,a_j^+).
 \label{well_d_2}
\end{align}
 Consider the terms at each point $x=a_j$: Since $|\cos(k \pi a_j)|\leq 1$ and using Corollary~\ref{well-ingham},
 there exist $C_1,C_2>0$, uniform in $t\in [0,T]$, such that, for $j\in\{0,...,N\}$, we have
\begin{align}
\sum_{k\in\N^*}\Big|\int_0^t e^{i\lambda_k\tau}
\cos(k \pi a_j)f (\tau,a_{j}^+)
d \tau\Big|^2&\leq C_1 \|\cos(k \pi a_j)f (\tau,a_{j}^+)
\|_{L^2((0,t),\C)}^2\nonumber\\
&\leq C_2 \|f \|_{L^2((0,t),C ((a_{j},a_{j+1}),\C))}^2\nonumber\\
&\leq C_2 \|f \|_{L^2((0,t),H^1(A) )}^2
\label{well_d_3}
\end{align}
and similarly, for $j\in\{1,...,N+1\}$,
\begin{align*}
\sum_{k\in\N^*}\Big|\int_0^t e^{i\lambda_k\tau} \cos(k \pi a_j)f (\tau,a_{j}^-) d \tau\Big| ^2&\leq C_2 \|f \|_{L^2((0,t),C ((a_{j-1},a_{j}),\C))}^2\nonumber\\
&\leq C_2 \|f \|_{L^2((0,t), H^1(A) )}^2.
\end{align*}
Therefore, there exists $C_3>0$, uniform in $t\in [0,T]$, such that
 \begin{align*}
\|G (t)\|_{H^1_0}^2
&= \sum_{k\in\N^*}\Big|\int_0^t e^{i\lambda_ks }\,k \lag \phi_k,f (s)\rag_{L^2} d s\Big|^2
\\
&\leq C_3\Big(\sum_{k\in\N^*}\Big|\int_0^t e^{i\lambda_k\tau}\lag \cos(k \pi x), g (\tau,\cdot)\rag_{L^2} d \tau\Big|^2\nonumber\\
&+ \|f \|_{L^2((0,t), H^1(A) )}^2\Big).
\end{align*}
Since for any $h\in L^2(0,t)$,
$$\big|\int_0^t e^{i\lambda_k\tau}h(\tau) d \tau\big|^2\leq t\int_0^t|h(\tau)|^2 d \tau$$
there exists $C_4>0$, uniform in $t\in [0,T]$, such that
\begin{align}
\|G (t)\|_{H^1_0}^2&\leq C_3\Big( t \sum_{k\in\N^*}\int_0^t |\lag \cos(k \pi x), g (\tau,\cdot)\rag_{L^2}^2 d \tau+ \|f \|_{L^2((0,t), H^1(A) )}^2\Big)\nonumber\\
&\leq C_3\Big(t \|g \|_{L^2((0,t), L^2((0,1),\C))}^2+ \|f \|_{L^2((0,t), H^1(A) )}^2\Big)\nonumber\\
&\leq C_4 \|f \|_{L^2((0,t), H^1(A) )}^2.
\label{bound_4}
\end{align}
Therefore, $G \in L^\infty( [0,T],H^1_0)$. Applying the same proof to
\[
G(t')-G(t)=\int_t^{t'}e^{-i\Delta s} f(s)\,\mathrm{d}s
\]
with $0\leq t \leq t' \leq T$, we get
\[
\|G(t')-G(t)\|_{H^1_0}^2\leq C_4 \|f \|_{L^2((t,t'), H^1(A) )}^2.
\]
Hence $G$ is continuous from $[0,T]$ to $H^1_0$.
\end{proof}

\begin{proof}[Proof of Theorem~\ref{well_d}]
Theorem~\ref{well_d} follows from  
the fact that $(e^{i\Delta t})_{t\in \R}$ is a continuous unitary group on $H^1_0$
and
Lemma~\ref{wellposedness-bound-3.d},
by applying the usual fixed point argument to the mapping $L: C( [0,T] ,H^1_0 ) \to C( [0,T] ,H^1_0 )$ defined by
\begin{equation*}
	L(\psi)(t)=e^{i\Delta t} \psi_0-i\int_0^te^{i\Delta(t-s)} \left(u_1(s)\mu \psi(s)  + f(s)\right) d s,
\end{equation*}
for $\psi \in C( [0,T],H^1_0)$.
\end{proof}

Recall that $\phi_l(t)=e^{-i\lambda_l t}\phi_l$ for every $l\in\N^*$. For $T>0$, we consider the orthogonal projector, for the underlying real structure,
$$P_T^l:L^2((0,1),\C)\rightarrow \TT_{\phi_l(T)},$$
where
\begin{align*}
  \TT_{\phi_l(T)}=&\Big\{\psi\in L^2((0,1),\C): \Re\big(\lag\psi,\phi_l(T) \rag_{L^2}\big)=0\Big\}\\
  =&\left\{\psi\in L^2((0,1),\C): \Re\Big(\int_0^1\psi(x)\phi_l(T,x)dx\Big)=0\right\},
  \end{align*}
  is the tangent space to the unit sphere $\sS$ in $L^2((0,1),\C)$ at $\phi_l(T)$. We introduce the end-point map of the solution to~\eqref{0.1_d} as follows:
$$\Psi_T^l:u\in L^2((0,T),\R)
\mapsto P^l_T\big(\psi(T)\big)\in H^1_0\cap \TT_{\phi_l(T)}$$
where $\psi$ is solution to~\eqref{0.1_d} with initial state $\psi_0=\phi_l$ and control $u$.
We are finally ready to ensure the $C^1-$regularity of the end-point map in the space $H^1_0$.

 \begin{proposition} \label{smooth_d}
Let $l\in\N^*$. Let $u\in L^2((0,T),\R) $ and $\psi$ the corresponding solution to~\eqref{0.1_d} in $(0,T)$ with initial state $\phi_l$. The mapping~$\Psi_T^l$  is $C^1$, 
we have $$\partial_{u} \Psi_T^l( u): L^2((0,T),\R)\rightarrow H^1_0\cap \TT_{\phi_l(T)}$$
and for any $v\in L^2((0,T),\R)$, 
$\partial_{u} \Psi_T^l( u)\cdot v=P_T^l(\xi(T)) $, where $\xi$ is the
linearised mild
solution around the solution $\psi$:
\begin{equation*}
 \xi(t)=\int_0^t e^{i \Delta (t-s)}\left(u(s)\mu \xi(s)+ v(s)\mu \psi (s)\right)\, \mathrm{d}s.
\end{equation*}
\end{proposition}

Proposition~\ref{smooth_d} is proved similarly to~\cite[Proposition 3]{BL-2010}, with the adaptation of the space to $H^1_0$ and the application of Theorem~\ref{well_d}. For completeness, we provide the proof in Appendix~\ref{endpoint}.

\subsection{Local exact controllability}
In this part, we focus on the local exact controllability, and we obtain the second point of Theorem~\ref{theorema} as a corollary of the following theorem.
\begin{theorem}\label{local-exact-d}
Let $N\in\N^*$ and $A:=\{a_j\}_{j\leq N}\subset (0,1)$. Let $l\in\N^*$ with
\begin{align*} j^2 - l^2\neq l^2-k^2,\qquad  \forall j,k\in\N^*\setminus\{l\}.\end{align*}
Let  $\mu\in H^1(A)$ such that
\begin{equation*}
\exists C>0, \forall k\in\N^*,
|\lag \mu\phi_l , \phi_{k}\rag_{L^2}|	\ge \frac{C}{k}.
\end{equation*} 
For any $T>0$, there exists $\delta>0$ such that, for~any~$\psi_0, \psi_1 \in H^1_0$ with $\|\psi_0\|_{L^2}=\|\psi_1\|_{L^2}=1$ and
\begin{equation*}
 \|\psi_0-\phi_l\|_{H^1_0}<\delta, \qquad \qquad \qquad  \|\psi_1-\phi_l(T)\|_{H^1_0}<\delta,
\end{equation*}
 there exists a control $u\in L^2((0,T),\R)$ such that the corresponding solution $\psi$ to the problem~\eqref{0.1_d} with $\psi(0)=\psi_0$ satisfies
 $$\psi(T)= \psi_1.$$
\end{theorem}
Theorem~\ref{theoremb} is proved by using the Inverse Mapping Theorem (more precisely the generalised version in the context of Hilbert spaces). To this purpose, we need to ensure the global exact controllability of the linear Schrödinger equation obtained by linearising~\eqref{0.1_d}
around the solution $\phi_{l}(t)$.
As in Proposition~\ref{smooth_d}, the linearised system corresponding to the control $u=0$ and initial state $\phi_l$ reads formally:
\begin{equation}
\label{lin_dbis}
\begin{cases}
i \p_t \xi=-\Delta \xi + v(t)\mu (x) \phi_l(t), & \quad x\in (0,1),\ t>0,\\
\xi(t,0)=\xi(t,1)=0, & \quad t>0,\\
\xi(0,x)=0, & \quad x\in (0,1).\\
\end{cases}
\end{equation}
We now state the global exact controllability of the linearised problem~\eqref{lin_dbis} in the following proposition.
\begin{proposition}\label{lin-exact-d}
Let
$N\in\N^*$
and
$A:=\{a_j\}_{j\leq N}\subset (0,1)$.
Let $l\in\N^*$ with
\begin{align}\label{resonant_d}
j^2 - l^2\neq l^2-k^2,\qquad  \forall j,k\in\N^*\setminus\{l\}.
\end{align}
Let $\mu\in H^1(A)$
such that
\begin{equation*}
\exists C>0,\,
\forall k\in\N^*,\,
|\lag \mu \phi_l, \phi_{k}\rag_{L^2}|	\ge \frac{C}{k}.
\end{equation*}
For any $T>0$ and $\psi_1 \in H_{0}^1\cap \TT_{\phi_l(T)}$, there exists a control $u\in L^2((0,T),\R)$ such that the corresponding mild solution $\xi$ to the problem~\eqref{lin_dbis} with $\xi(0)=0$ satisfies  $$\xi(T)= \psi_1.$$
\end{proposition}

\begin{proof} Let us start by considering the case $l=1$. We note that 	$$ 	 	 \xi(t)= -i\int_0^te^{i\Delta(t-s)}  v(s) \mu(x) \phi_1(s) d s, \quad t\in [0,T].$$
	We decompose $\xi(T)$ by using the basis $\{\phi_k\}_{k\in\N^*}$ and we obtain
$$\xi(T)=-i\sum_{k\in\N^*}e^{-i\lambda_{k} T} \int_0^T e^{i (\lambda_{k}-\lambda_1) s} v (s)\lag \mu\, \phi_1 ,\phi_{k}\rag_{L^2}\phi_k d s.$$
Now, we project $\xi(T)$ on the elements of basis and, for every $k\in\N^*$,
$$
 \lag \xi(T),\phi_{k}\rag_{L^2}=-i e^{-i\lambda_{k} T}\lag \mu\,\phi_1 ,\phi_k\rag_{L^2}\int_0^T e^{i (\lambda_{k}-\lambda_1) s} v(s) d s,$$
which implies
\begin{align}
\int_0^T e^{i (\lambda_{k}-\lambda_1) s} v(s) ds &= \frac{i e^{i\lambda_{k} T}\lag \xi(T),\phi_{k}\rag_{L^2}}{\lag \mu\,\phi_1 ,\phi_k\rag_{L^2}} ,\qquad  \forall k\in\N^*. \label{MP_d}
\end{align}
Note that, for any $ \xi \in H^1_0\cap \TT_{\phi_1(T)}$, since $\lag \mu ,\phi_k\rag_{L^2}\geq C k^{-1}$ for $C>0$, we have
$$\left( \frac{i e^{i\lambda_{k} T}\lag \xi(T),\phi_{k}\rag_{L^2}}{\lag \mu \phi_1,\phi_k\rag_{L^2}} \right)_{k\in\N^*}\in \ell^2_r(\N^*,\C):=\left\{(x_k)_{k\in \N^*}\in\ell^2:\, x_1\in\R\right\}.$$
For every fixed $T>0$, thanks to Lemma~\ref{asympt_spec}, we apply Proposition~\ref{solvability_moment_problem} to the moment problem~\eqref{MP_d} by considering the sequence $(\lambda_k-\lambda_1)_{k\in\N^*}$. Then, there exists $v\in L^2((0,T),\R)$ such that
$$
\int_0^T e^{i\lambda_{k} s} v(s) d s = \frac{i e^{i\lambda_{k} T}\lag \xi(T),\phi_{k}\rag_{L^2}}{\lag \mu\phi_1 ,\phi_k\rag_{L^2}}, \quad k\in\N^*.
$$
This shows that the mapping $v\in L^2((0,T),\R)\mapsto \xi(T)\in H^1_0\cap T_{\phi_1(T)}$ is surjective and completes the proof when $l=1$. The general case with $l \in \N^*$ is treated in the same way as the previous one, by considering the frequencies $(\lambda_k - \lambda_l)_{k \in \N^*}$. Note that such a sequence satisfies the hypotheses of~\eqref{solvability_moment_problem} thanks to~\eqref{resonant_d}.
\end{proof}
 \begin{proof}[Proof of Theorem~\ref{local-exact-d}]
 	We know that $\phi_l(t)$ is a solution to the problem~\eqref{0.1_d} corresponding to the control $u(t)=0 $. Theorem~\ref{well_d} and Theorem~\ref{smooth_d} imply that the operator $\Psi_T^l:u \mapsto \psi(T)$ is well-defined and $C^1$-smooth in $H^1_0$. Furthermore, for   any $v\in L^2((0,T),\R)$, we have
  $\partial_u\Psi_T^l(u=0)\cdot v=\xi(T)$, where $\xi$ is the mild solution to~\eqref{lin_dbis}. The conditions of Proposition~\ref{lin-exact-d} are satisfied, so the linearised system is exactly controllable. Applying the Inverse Mapping Theorem to the mapping $u\mapsto \Psi_T^l(u)$ and the time-reversibility property of the Schrödinger  equation, we complete the proof.
 \end{proof}

\section{The periodic case}\label{section_p}
In this section, we study the bilinear Schrödinger equation~\eqref{0.1_p} and prove Theorem~\ref{theoremb}. In the first part, we provide some spectral properties for the operator $-\Delta + u_0P$ with $P = i\p_x$ and the well-posedness of the equation in $H^1_p$. Then, we prove the local exact controllability in the same space.

Note that~\eqref{0.1_p} (and then~\eqref{0.1}) can be obtained by applying 
to the Schrödinger equation~\eqref{intro_eq} a gauge transformation. For 
any $f$ smooth, the Laplacian operator is modified to a kind of magnetic Laplacian by the formula 
$$e^{i f(x)}\Delta e^{-i f(x)}~=~ \big(\p_x-i\p_x f\big)^2.$$ Now, the function $\psi= e^{i f(x)}\varphi$, with $\varphi$ solutions of~\eqref{intro_eq}, solves the equation
\begin{equation}\label{0.2}
i \p_t \psi =-\big(\p_x+iA\big)^2 \psi + u(t)\mu(x)  \psi, \quad x\in (0,1),\ t>0,
\end{equation}
where $A=-\p_x f$ is, in a sense, a magnetic field arising from the change of gauge. We consider $f(x)=u_0\ x$ so that $A(x)= -u_0.$ In this case, we can rewrite~\eqref{0.2} in a way very similar to~\eqref{0.1_p}, {\it i.e. }
\begin{equation*}
i \p_t \psi =(-\Delta+ 2u_0 P - u_0^2 ) \psi  +u(t)\mu(x) \psi, \quad x\in (0,1),\ t>0.
\end{equation*}
 From this perspective, the initial problem~\eqref{intro_eq} can be rewritten as a new Schrödinger equation in the presence of a momentum operator, as in~\eqref{0.1_p}. However, the gauge transformation alters the periodic boundary conditions when $u_0\not\in 2\pi \Z$.

\subsection{Spectral properties}

Let us start by denoting, since $u_0$ will be fixed, the principal operator appearing in~\eqref{0.1_p} as
$$
H = -\Delta  + u_0 P.
$$
The operator $H$ is self-adjoint in the domain
$$H^2_p:=\left\{\psi\in H^2((0,1),\C)\,:\, \psi(0)=\psi(1),\, \p_x\psi(0)=\p_x\psi(1)\right\}.$$
and has a purely discrete spectrum given by the eigenvalues
$$
\lambda_k = 4 \pi^2 k^2 - 2 \pi u_0 k, \quad \forall k \in \Z.
$$
The eigenvalues correspond to the eigenfunctions with $k\in\Z$
$$
\phi_k = e^{i 2 k \pi x}, \quad x \in (0,1).
$$
Let $l\in\Z$, define $(\nu_k^l)_{k\in\Z}$ as the increasing sequence such that 
$$\left\{\nu_k^l\mid\,k\in\Z\right\}=\left\{\lambda_l-\lambda_{k},\ 0,\, \lambda_k-\lambda_l\mid\,k\in \Z^*\right\}.$$
Note that in general there may exist $k$ and $j$ distinct such that $\lambda_l-\lambda_{k}=\pm \lambda_l-\lambda_{j}$. Below, we show that for some suitable $u_0$ this is not possible.
\begin{lemma}\label{asympt_spec}
Let $u_0\not\in 2\pi \Q$ and $l\in\Z$. The sequence $(\nu_k^l)_{k\in\Z}$ is such there exists $C>0$
\[
|\nu_{k}^l-\nu_{k-1}^l|\geq C |k|,\quad \forall k\in \Z
\]
and
\begin{equation*}
\inf_{n, m\in\Z} |\nu_{n}^l-\nu_m^l|>0.
\end{equation*}
\end{lemma}
\begin{proof}
Let us first consider $l=0$.  
For every $m,n\in\Z$ 
\begin{align}
\lambda_n\pm \lambda_m&= 2\pi^2(2\pi (n^2\pm m^2)- {u_0}(n\pm m))\nonumber\\
&=\begin{cases}2\pi^2\big(2\pi (n^2+ m^2)- {u_0}(n+ m)\big),\\
2\pi^2\big(2\pi (n+ m)- {u_0}\big)(n- m).\end{cases}
\label{ll}
\end{align}
The choice of $u_0$ ensures that for distinct $m,n\in\Z$ 
$\lambda_n$ and $\lambda_m$ are distinct
Then note that, for distinct $m,n\in\Z^*$ such that $|m|,|n|$ are large enough, we have $|\lambda_n\pm \lambda_m|\geq C(|n|+|m|)$ and then $|\nu_{k}^0-\nu_{k-1}^0|\geq C|k|$ for some $C>0$ with $|k|$ sufficiently large.
This prove the result when $l=0$. 

The general case $l\in\Z$ is proved equivalently by considering that
$$(\lambda_n-\lambda_l)\pm (\lambda_m-\lambda_l)= 2\pi^2\Big(2\pi \big(n^2\pm m^2-(l^2\pm l^2)\big)- {u_0}\big(n\pm m -(l\pm l)\big)\Big).$$ The last quantity is always different from $0$ when $n,m\in \Z^*\setminus\{l\}$. Indeed, either we obtain one of the identities in~\eqref{ll}, or
$$2\pi^2\big(2\pi (n^2+ m^2-2 l^2)\big)- {u_0}\big(n+ m - 2l)\big)$$
which is equal to $0$ only when $(n-m)^2= 0$ and $2l= n+m$ for every $m,n\in\Z\setminus\{l\}$. The last relations are both valid only when $m=n=l$ and then $(\nu_k^l)_{k\in \Z}$ is composed by distinct numbers.
The remaining part of the proof follows as in the case $l=0$.
\end{proof}

Lemma~\ref{asympt_spec} implies the simplicity of the spectrum of the operator $H$ for specific $u_0$, in contrast to the Laplacian with periodic boundary conditions, which exhibits double eigenvalues. The presence of $P$ induces a Zeeman-type effect on the spectrum, allowing us to avoid the occurrence of double eigenvalues.

\subsection{Well-posedness}

Recall that $H= -\Delta +u_0 P$ and we consider the following linear Schrödinger equation with a source term:
\begin{equation}
\label{0.1_p_bis}
\begin{cases}
 i \p_t \psi =H\psi +u_1(t)\mu(x)\psi+ f(t,x), & \quad x\in (0,1),\ t>0,\\
 \psi(t,x=0)=\psi(t,x=1),& \quad t>0,\\
 \p_x\psi(t,x=0)=\p_x\psi(t,x=1),& \quad t>0,\\
 \psi_0(t=0)=\psi_0\in L^2((0,1),\C),
\end{cases}
\end{equation}

We now present the well-posedness of~\eqref{0.1_d_bis} (and then~\eqref{0.1_p}) to ensure the first point of Theorem~\ref{theoremb}.
\begin{theorem}\label{well_p}
Let $N \in \N^*$, $A = \{a_j\}_{j \leq N} \subset (0,1)$ and $T>0$. 
Let $\mu\in H^1(A)$, $\psi_0\in H^{1}_p$, and $f \in L^2( (0,T) , H^{1}(A))$.
Then, for any $u\in L^2( (0,T) ,\R)$, there exists a unique mild solution $\psi \in C( [0,T] ,H^{1}_p)$ to the problem~\eqref{0.1_p_bis} such that $\psi(0)=\psi_0$ that is a solution to
$$
 \psi(t)=e^{-iH t} \psi_0-i\int_0^te^{-iH(t-s)}(u(s)\mu \psi(s) + f(s))\,\mathrm{d}s, \quad t\in [0,T] .
$$ Moreover, there is a constant~$C_T>0$ such that
$$\|\psi\|_{C( [0,T] ,H^{1}_p )}\le C_T \left (\|\psi_0\|_{H^1_p}+\|f\|_{L^2( (0,T), H^1(A))}\right).$$
\end{theorem}

The proof of Theorem~\ref{well_p} is based on the same strategy of the proof of Theorem~\ref{well_d}. We just retrace the main steps for the sake of completeness. As in the previous case, we need to introduce the following intermediate lemma.

\begin{lemma}\label{wellposedness-bound-3.p}
Let $N \in \N^*$, $A = \{a_j\}_{j \leq N} \subset (0,1)$, and $f\in L^2( (0,T) , H^{1}(A))$. 
Then, the function $G:t\mapsto \int_0^te^{iHs} f(s)\,\mathrm{d}s$ belongs to $C( [0,T] ,H^1_p)$ and there is a constant $C_T > 0$ such that
\begin{equation*}
\|G \|_{C( [0,T] ,H^1_p)}\le C_T \|f\|_{L^2( (0,T) , H^1(A))}.
\end{equation*}
\end{lemma}
\begin{proof} 
The proof follows the same strategy as for Lemma~\ref{wellposedness-bound-3.d} and uses the spectral decomposition $G(t)=\sum_{k\in\Z}\phi_k \int_0^t e^{i\lambda_k\tau}\lag \phi_k,f(\tau)\rag_{L^2} \,\mathrm{d}\tau$. Recall that
\begin{equation*}
\|G(t)\|_{ H^1_p }^2= \Big|\int_0^t e^{i\lambda_0s}\lag \phi_0,f (s)\rag_{L^2}\,\mathrm{d}s\Big|^2 + \sum_{k\in\Z^*}\Big|\int_0^t e^{i\lambda_ks}k \lag \phi_k,f (s)\rag_{L^2}\,\mathrm{d}s\Big|^2.
\end{equation*}
The central part consists in estimating $k\lag \phi_k, f(\tau, \cdot)\rag_{L^2}$ for each $k\in\Z^*$ for which the relation~\eqref{well_d_2} is verified. We estimate the boundary terms again as in~\eqref{well_d_3} by using Corollary~\ref{well-ingham}. As in~\eqref{bound_4}, there exists $C > 0$ uniform in $[0, T]$ such that $\|G (t)\|_{H^1_p}\leq C \|f \|_{L^2((0,t), H^1(A)}.$ Finally, $G \in C( [0,T], H^1_p)$ thanks to absolute continuity of the integral.
\end{proof}

\begin{proof}[Proof of Theorem~\ref{well_p}]
The proof exploits the same fixed-point arguments introduced in the proof of Proposition~\ref{well_d}. The only difference is that we need to substitute the space $H^1_0$ with $H^1_p$ and use Lemma~\ref{wellposedness-bound-3.p} instead of Lemma~\ref{wellposedness-bound-3.d}.
\end{proof}

Recall that $\phi_l(t)=e^{-i\lambda_l t}\phi_l$ with $l\in\Z$. We introduce for $T>0$ the orthogonal projector 
for the underlying real structure
$P_T^l:L^2((0,1),\C)\rightarrow \TT_{\phi_l(T)}$ where
\begin{align*}
 \TT_{\phi_l(T)}=&\left\{\psi\in L^2((0,1),\C): \Re(\lag \psi,\phi_l(T)\rag_{L^2})=0\right\}\\
 =&\left\{\psi\in L^2((0,1),\C): \Re\Big(\int_0^1\psi(x)\phi_l(T,x)dx\Big)=0\right\},
 \end{align*}
is the tangent space to the unit sphere $\sS$ in $L^2$ at $\phi_l(T)$. Let the end-point map of the solution to~\eqref{0.1_p} as follows:
$$\Psi_T^l:u\in L^2((0,T),\R)
\mapsto P^l_T\big(\psi(T)\big)\in H^1_p\cap \TT_{\phi_l(T)},$$
where $\psi$ is solution to~\eqref{0.1_p} with initial state $\psi_0=\phi_l$ and control $u$. We present the $C^1-$regularity of the end-point map in the space $H^1_p$ in the following propositions, which is proven as Proposition~\ref{smooth_d} as explained in Appendix~\ref{endpoint}

\begin{proposition} \label{smooth_p}
Let $l\in\Z$. Let $u\in L^2((0,T),\R) $ and $\psi$ the corresponding solution to~\eqref{0.1_p} in $(0,T)$ with initial state $\phi_l$. Then, the mapping~$\Psi_T^l$ is $C^1$, and for any $v\in L^2((0,T),\R)$, we have $$\partial_{u} \Psi_T^l( u): L^2((0,T),\R)\rightarrow H^1_p\cap \TT_{\phi_l(T)}$$
and $\partial_{u} \Psi_T^l( u)\cdot v=P_T^l(\xi(T)) $, where $\xi$ is the
linearised mild
solution
\begin{equation*}
 \xi(t)=\int_0^t e^{-i H (t-s)}\left(u(s)\mu \xi(s)+ v(s)\mu \psi (s)\right)\, \mathrm{d}s.
\end{equation*}
\end{proposition}

\subsection{Local exact controllability}
We are finally ready to prove the local exact controllability of the problem~\eqref{0.1_p} and prove the second point of Theorem~\ref{theoremb}.
\begin{theorem}\label{local-exact-p}
Let $N\in\N^*$ and $A:=\{a_j\}_{j\leq N}\subset(0,1)$. Let $u_0\in \R\setminus 2\pi \Q$ and $l\in\Z$. Let $\mu\in H^1(A)$ such that 
\begin{equation*}
\exists C>0,\, \forall k\in \Z,\,|\lag \mu\phi_l , \phi_{k}\rag_{L^2}|\ge \frac{C}{|k|+1}.
\end{equation*} 
Then, for any $T>0$, there exists $\delta>0$ such that, for~any~$\psi_0, \psi_1 \in H^1_p$ with $\|\psi_0\|_{L^2}=\|\psi_1\|_{L^2}=1$ and
\begin{equation*}
 \|\psi_0-\phi_l\|_{H^1_p}<\delta, \quad \quad \|\psi_1-\phi_l(T)\|_{H^1_p}<\delta,
\end{equation*}
 there exists a control $u_1\in L^2((0,T),\R)$ such that the corresponding solution $\psi$ to the problem~\eqref{0.1_p} with $\psi(0)=\psi_0$ satisfies $$\psi(T)= \psi_0.$$
\end{theorem}

One may wonder if the local exact controllability of~\eqref{0.1_p} can be ensured when $u_0 = 0$. The main problem is the presence of double eigenvalues for the spectrum of the periodic Laplacian, which prevents us from using the strategy developed below for the proof of Theorem~\ref{local-exact-p}. A possibility may be adding one or more controls to the equation. However, the local exact controllability is still not straightforward. We refer to Remark~\ref{periodic} for further discussion.

The proof of Theorem~\ref{local-exact-p} follows the approach leading to Theorem~\ref{local-exact-d}. The Inverse Mapping Theorem allows us to infer the result by proving the global exact controllability of the linearised problem. In detail, we consider $\psi = \phi_{l}(t)$, which is a solution to~\eqref{0.1_p} when $u = 0$ and $\psi_0 = \phi_l$. The linearised system in this case reads:
\begin{equation}
\label{lin_pbis}
\begin{cases}
 i \p_t \xi=H \xi + v(t)\mu (x)\phi_l(t) , & \quad x\in (0,1),\ t>0,\\
 \xi(t,0)=\xi(t,1), & \quad t>0,\\
 \p_x\xi(t,0)=\p_x\xi(t,1), & \quad t>0,\\
 \xi(0,x)=0, & \quad x\in (0,1).
 \end{cases}
 \end{equation}
We ensure now the global exact controllability of the Schrödinger equation~\eqref{lin_pbis}.
\begin{proposition}\label{lin-exact-p}
Let $N\in\N^*$ and $A:=\{a_j\}_{j\leq N}\subset(0,1)$. Let $u_0\in \R\setminus 2\pi \Q$ and $l\in\Z$. Let $\mu\in H^1(A)$ such that
\begin{equation*}
\exists C>0,\,\forall k\in\Z,\,|\lag \mu\phi_l , \phi_{k}\rag_{L^2}|\ge \frac{C}{|k|+1}.
\end{equation*}
Then, for any $T>0$ and $\psi_1 \in H_{p}^1\cap \TT_{\phi_l(T)}$, there exists a control $u\in L^2((0,T),\R)$ such that the corresponding mild solution $\xi$ to the problem~\eqref{lin_pbis} with $\xi(0)=0$ satisfies $$\xi(T)= \psi_1.$$
\end{proposition}

\begin{proof} The proof follows the strategy of Theorem~\ref{lin-exact-d}. Let us consider the case $l=0$ so that $\phi_l(t)=\phi_0(t)=1$. We write $ \xi(t)= -i\int_0^te^{-iH(t-s)} v(s) \mu(x)\,\mathrm{d}s$ and we decompose $\xi(T)$
$$
 \lag \xi(T),\phi_{k}\rag_{L^2}=-i e^{-i\lambda_{k} T}\lag \mu ,\phi_k\rag_{L^2}\int_0^T e^{i \lambda_{k} s} v(s)\,\mathrm{d}s,\qquad  \forall k\in\Z,$$
which leads to the moment problem $
\int_0^T e^{i (\lambda_{k}-\lambda_0) s} v(s) ds = \frac{i e^{i\lambda_{k} T}\lag \xi(T),\phi_{k}\rag_{L^2}}{\lag \mu ,\phi_k\rag_{L^2}}$ for every $\forall k\in\Z.$ Now, we have $\lag \mu ,\phi_k\rag_{L^2}\geq C (1+|k|)^{-1}$ for $C>0$ and, for any $ \xi \in H^1_p\cap \TT_{\phi_0(T)}$,
$$\left( \frac{i e^{i\lambda_{k} T}\lag \xi(T),\phi_{k}\rag_{L^2}}{\lag \mu ,\phi_k\rag_{L^2}} \right)_{k\in\Z}\in \ell^2_r(\Z,\C):=\left\{(x_k)_{k\in \Z}\in\ell^2:\, x_0\in\R\right\}.$$
For every fixed $T>0$, thanks to Lemma~\ref{asympt_spec}, we apply Proposition~\ref{solvability_moment_problem} to ensure the solvability of the moment problem corresponding to the control issue, hence there exists $v\in L^2((0,T),\R)$ such that
$$
\int_0^T e^{i\lambda_{k} s} v(s)\,\mathrm{d}s = \frac{i e^{i\lambda_{k} T}\lag \xi(T),\phi_{k}\rag_{L^2}}{\lag \mu ,\phi_k\rag_{L^2}}, \quad k\in\Z.
$$
Finally, the global exact controllability is ensured when $l=0$. The general case $l\neq 0$ follows equivalently by the solvability of the moment problem
$$
\int_0^T e^{i(\lambda_{k}-\lambda_{l}) s} v(s)\,\mathrm{d}s = \frac{i e^{i\lambda_{k} T}\lag \xi(T),\phi_{k}\rag_{L^2}}{\lag \mu\phi_l ,\phi_k\rag_{L^2}}, \quad k\in\Z.
$$
The result is again proved by using Proposition~\ref{solvability_moment_problem} thanks to Lemma~\ref{asympt_spec}.\end{proof}

\begin{proof}[Proof of Theorem~\ref{local-exact-p}]
We follow the proof of Theorem~\ref{local-exact-d}. We know that $\phi_l(t)$ is a solution to the problem~\eqref{0.1_p} corresponding to the control $u(t) = 0$. We use the strategy leading to Theorem~\ref{local-exact-p}: the operator $\Psi_T^l: u \mapsto \psi(T)$ is $C^1$-smooth, and for any $v \in L^2((0, T), \R)$, we have $\partial_u \Psi_T^l(u=0) \cdot v = \xi(T),$ where $\xi$ is the solution to~\eqref{lin_pbis}. Now, Proposition~\ref{lin-exact-p} establishes the global exact controllability of the linearised system~\eqref{lin_pbis}. Finally, the result follows by applying the Inverse Mapping Theorem to the mapping $u \mapsto \Psi_T^l(u)$ and using the time-reversibility property of the Schrödinger equation.
\end{proof}


\section{The Neumann case}\label{comments_n}

In Sections~\ref{Dirichlet} and~\ref{section_p}, we developed a strategy leading to the local exact controllability in suitable subspaces of $H^1$ of the bilinear Schrödinger equations~\ref{0.1_d} and~\ref{0.1_p}, respectively. It is natural to wonder if the same strategy applies to other boundary conditions such as the Neumann ones. Recall that the controllability and the well-posedness can be ensured in $D(-\Delta)$ where $-\Delta$ is the Neumann Laplacian as soon as $\mu\in H^2$, see~\cite{BL-2010}. This section presents the application of our strategy to study this framework in $H^1$ and some open problems.

\subsection{Well-posedness}

We consider the bilinear Schrödinger equation in the presence of Neumann boundary conditions:
\begin{equation}\label{0.1_n}
\begin{cases}
 i \p_t \psi =-\Delta\psi+ u(t)\mu(x)\psi, & \quad x\in (0,1),\ t>0,\\
 \p_x\psi(t,x=0)=\p_x\psi(t,x=1)=0,& \quad t>0,\\
 \psi_0(t=0)=\psi_0\in L^2((0,1),\C).
\end{cases}
\end{equation}
The Neumann Laplacian in~\eqref{0.1_n} is self-adjoint and with compact resolvent. Its spectrum is made of isolated and simple eigenvalues $(\lambda_k)_{k\in\N}$ such that $\lambda_k={k^2\pi^2}.$ A family of corresponding eigenfunctions, thus forming a Hilbert basis for $L^2$, is given by
$$\phi_0(x)=1,\qquad\phi_k(x)= \sqrt{2}\cos\big({k}\pi x\big),\qquad \forall k\in\N^*.$$
To consider the well-posedness of~\eqref{0.1_n} we introduce the corresponding bilinear Schrödinger equation in the presence of a source term:
\begin{equation}\label{0.1_n_bis}
\begin{cases}
 i \p_t \psi =-\Delta\psi+ u(t)\mu(x)\psi + f(t,x) & \quad x\in (0,1),\ t>0,\\
 \p_x\psi(t,x=0)=\p_x\psi(t,x=1)=0,& \quad t>0,\\
 \psi_0(t=0)=\psi_0\in L^2((0,1),\C).
\end{cases}
\end{equation}

The well-posedness of~\eqref{0.1_n_bis} in the space $H^1$ (and then~\eqref{0.1_n}) is given by the following proposition.
\begin{proposition}\label{well_n}
Let
$N \in \N^*$,
$A = \{a_j\}_{j \leq N} \subset (0,1)$ and
$T>0$.
Let $\mu\in H^1(A)$, $\psi_0\in H^{1}$ and $f \in L^2( (0,T) , H^{1}(A))$.
For any $u\in L^2( (0,T) ,\R)$, there exists a unique mild solution $\psi \in C( [0,T] ,H^{1})$ to the problem~\eqref{0.1_n_bis} such that $\psi(0)=\psi_0$ that is a solution to
$$
\psi(t)=e^{i\Delta t} \psi_0-i\int_0^te^{i\Delta(t-s)}(u(s)\mu \psi(s) +  f(s))  d s, \quad t\in [0,T].
$$
Moreover, there is a constant~$C_T>0$ such that
$$\|\psi\|_{C( [0,T], H^{1} )}\le C_T \left (\|\psi_0\|_{H^1}+\|f\|_{L^2( (0,T) , H^{1}(A))}\right).$$
\end{proposition}
\begin{proof}
The proof is an adaptation of the one of Theorem~\ref{well_d} by observing that $D(|\Delta|^\frac{1}{2}) = H^1$ when $-\Delta$ is the Neumann Laplacian. The space $H^1$ is here endowed the norm 
$$\sqrt{|\lag\cdot,\phi_0\rag_{L^2}|^2+\sum_{j\in\N^*}|k\lag\cdot,\phi_k\rag_{L^2}|^2}.$$
The decomposition of the solution is with respect to the basis $\{\phi_k\}_{k \in \N}$, the eigenfunctions of $-\Delta$.
\end{proof}

\subsection{An open problem: the local exact controllability in \texorpdfstring{$H^1$}{H1}}

Proposition~\ref{well_n} ensures the well-posedness of~\eqref{0.1_n} in $H^1$ and the $C^1$-smoothness of the endpoint map in the same space can be proved as in the Dirichlet case (Proposition~\ref{smooth_d}). This framework is suitable for the Inverse Mapping Theorem, and one may wonder if the local exact controllability, as in Theorem~\ref{local-exact-d}, can be ensured nearby a bounded state $\phi_l$. The techniques leading to Theorem~\ref{local-exact-d} can be applied, at least from an abstract point of view, to the Neumann case.

In detail, let $l\in\N$ verifying~\eqref{resonant}. Let $\mu$ such that the well-posedness of~\eqref{0.1_n} and the $C^1$-smoothness of the endpoint map in $H^1$ are ensured. Assume also the existence of $C > 0$ so that
\begin{equation}\label{baddd}
|\lag \mu\phi_l, \phi_{k}\rag_{L^2}|\ge \frac{C}{|k|+1},\qquad \forall k\in\N.
\end{equation}
Then, the local exact controllability of~\eqref{0.1_n} holds, as for Theorem~\ref{local-exact-d}, around the bounded state $\phi_l$.

Here lies the main difficulty with this framework.  
There are no potentials $\mu$ satisfying both Proposition~\ref{well_n} and the relations~\eqref{baddd}, as illustrated in the following proposition. From this perspective, the strongest local exact controllability result in the Neumann case remains the one by Beauchard and Laurent in the domain of the operator $D(\Delta)$ with $\mu \in H^2(0,1)$, proved in \cite{BL-2010}.

\begin{proposition}
Let $l\in\N^*$. Assume $\mu \in  H^1(A)$ with $A = \{a_j\}_{j \leq N} \subset \R$ and $N \in \N^*$. There does not exist $C > 0$ such that
\begin{equation}\label{bad}
|\lag \mu\phi_l, \phi_{k}\rag_{L^2}|\ge \frac{C}{|k|+1},\quad \forall k\in\N.
\end{equation}
\end{proposition}
\begin{proof}
Let us start by studying the simpler case of $l=0$ and $\mu$ being a piecewise constant function. We treat the general case later on.

\smallskip

\noindent
{\bf 1. Piecewise constant function}. Let $l=0$. Assume $\mu = \sum_{l=1}^N \alpha_l 1_{[r_{2l-1},r_{2l}]},$ for some $\alpha_l \in \R$ and $ 0 \leq r_{2l-1}< r_{2l}\leq 1\in\R$ with $1\leq l\leq N$.
We have that
\begin{align*} \lag \mu, \phi_k \rag_{L^2} = \frac{\sqrt{2}}{k \pi}\sum_{l=1}^N\alpha_l\big(\sin(k\pi r_{2l})-\sin(k\pi r_{2l-1})\big),\quad \forall k\in\N^*.\end{align*}
Consider 
$$I := \left\{ j \in \{1, \ldots, 2N\}\mid\, r_l \in \R \setminus \Q\right\}$$ 
For every $l \in I$, we have $r_l = p_l / q_l$ for some $p_l, q_l \in \N^*$. For every $k$ such that, 
\begin{equation}\label{bad_2}
k \in Q \cdot \N^*, \quad\quad  Q := \prod_{l \in I} q_l \in \N^*,
\end{equation}
it holds $\sin(k\pi r_l) = 0$ for every $l \in I$. Then, for every $k \in Q \cdot \N^*$, we have
\begin{align*}
\lag \mu, \phi_k \rag_{L^2} = \frac{\sqrt{2}}{k \pi} \sum_{l \in I^c} (-1)^l \alpha_l\sin(k\pi r_l) = \frac{\sqrt{2}}{k \pi} \sum_{l \in I^c} (-1)^l \alpha_l \sin(n\pi \tilde r_l),
\end{align*}
where $n = \frac{k}{Q}$ and $\tilde r_l = Q\cdot r_l\in\R\setminus\Q$. Now, we show that~\eqref{bad} cannot be verified for every $k$ satisfying~\eqref{bad_2}, due to the theory of Diophantine approximations of irrational numbers (see, for instance, \cite[Appendix A]{GGCC}, ``The Graph Geometric Control Condition''). Indeed, from \cite[Theorem A.2]{GGCC}, as explained in \cite[relation (64); Appendix A]{GGCC}, there exist infinitely many $n, m \in \N^*$ such that, for every $l\in I^c$
\begin{equation}\label{approx}
|n \tilde r_l - m| = n \left| \tilde r_l - \frac{m}{n} \right| \leq n \frac{1}{n^{1+\frac{1}{card(I^c)}}}=  \frac{1}{n^{\frac{1}{card(I^c)}}},.
\end{equation}
The last inequality implies the existence of $C > 0$ such that
\begin{align*}
\frac{\sqrt{2}}{n Q  \pi} \left|\sum_{l \in I^c} (-1)^l \sin(n\pi \tilde r_l)\right| \leq \frac{C}{n^{1+\frac{1}{card(I^c)}}},
\end{align*}
for those $n$ satisfying~\eqref{approx}. This fact is an obstruction to the validity of~\eqref{bad} for every $k$ satisfying~\eqref{bad_2}, which concludes the proof in this case.

\smallskip

\noindent
{\bf 2. Piecewise $H^1$ function}. Consider now $l=0$ and $\mu \in H^1(A)$. We have  
$$
\lag \mu, \phi_k \rag_{L^2} = -\frac{\sqrt{2}}{k \pi} \sum_{l=1}^N \big(\mu'(a_l^+) \sin(k\pi a_l^+) - \mu'(a_l^-) \sin(k\pi a_l^-)\big) - \frac{\sqrt{2}}{k \pi} \lag \nu, \sin(k\pi x) \rag_{L^2},
$$  
for $\nu := \big(\p_x \mu\big) 1_{(0,a_1)} + \ldots + \big(\p_x \mu\big) 1_{(a_N,1)} \in L^2\big((0,1),\C\big)$.
Since $\{\sqrt{2} \sin(k\pi x)\}_{k \in \N^*}$ is a Hilbert basis of $L^2((0,1),\C)$, we have  
\begin{equation}\label{obstr1}\lag \nu, \sin(k\pi x) \rag_{L^2} \xrightarrow[]{k\rightarrow+\infty}0.\end{equation}
Finally, the arguments used in Part {\bf 1.} of this proof show the existence of $C,p > 0$ such that
$$
\frac{\sqrt{2}}{k \pi}\left| \sum_{l=1}^N \big(\mu'(a_l^+) \sin(k\pi a_l^+) - \mu'(a_l^-) \sin(k\pi a_l^-)\big) \right| \leq \frac{C}{k^{1+p}},
$$  
for infinitely many $k \in \N^*$. This property and the asymptotic behaviour~\eqref{obstr1} prevent the validity of~\eqref{bad} and conclude the proof when $l=0$. The general case of $l\neq 0$ is proved in the same way since $\mu\phi_l \in H^1(A)$ when $\mu \in H^1(A)$.
\end{proof}

\begin{remark}\label{periodic}
Note that, if we were able to prove the local exact controllability of~\eqref{0.1_n} in $H^1$, it would be possible to ensure Theorem~\ref{theoremb} with two suitable control potentials $\mu$ and without using the momentum operator $P$ (so when $u_0 = 0$). Indeed, consider the equation
\begin{equation}\label{0.1_p_twocontrols}
 i \p_t \psi =-\Delta\psi+ u_1(t)\mu_1(x)\psi+ u_2(t)\mu_2(x)\psi,  \quad x\in (0,1),\ t>0,
\end{equation}
in the presence of periodic boundary conditions and with $\mu_1$ and $\mu_2$ respectively antisymmetric and symmetric with respect to the point $x=1/2$. For every $f:(0,1)\rightarrow \C$, denote $f^1(x)=\frac{f(x)-f(1-x)}{2}$ and $f^2= \frac{f(x)+f(1-x)}{2},$
the antisymmetric and the symmetric part of $f$, respectively.
Now, the Schrödinger equation appearing in~\eqref{0.1_p_twocontrols} can be rewritten as a system of coupled equations, respectively in the space of those $L^2$ functions which are antisymmetric and symmetric:
$$ i \p_t \psi^1 =-\Delta\psi^1+ u_1(t)\mu_1(x)\psi^2+ u_2(t)\mu_2(x)\psi^1,$$
$$ i \p_t \psi^2 =-\Delta\psi^2+ u_1(t)\mu_1(x)\psi^1+ u_2(t)\mu_2(x)\psi^2.$$
The linearization of~\eqref{0.1_p_twocontrols} with respect to the constant function $1$ consists in two decoupled linear equations: one in the presence of Dirichlet boundary conditions, and the second with Neumann boundaries. For this reason, if we were able to prove the local exact controllability in both the Dirichlet case in $H^1_0$ and the Neumann one in $H^1$, we would be able to prove the same controllability of the periodic Schrödinger equation~\eqref{0.1_p_twocontrols} in $H^1_p$.
\end{remark}

\section{Quantum Harmonic oscillator}\label{comments_h}
Let us now discuss the application of the above strategy to the quantum Harmonic oscillator. We consider the following bilinear Schrödinger equation in $L^2(\R,\C)$:
\begin{equation}\label{0.1_h}\begin{cases}
 i \p_t \psi =[-\Delta+x^2]\psi+ u(t)\mu(x)\psi & \quad x\in \R,\ t>0,\\
 \psi_0(t=0)=\psi_0\in L^2(\R,\C).
\end{cases}
\end{equation}

\subsection{Spectral properties}

Let us introduce $H=-\Delta +x^2$ defined in the domain
$$D(H)=\{\psi\in L^2(\R,\C)\,:\, (-\Delta+ x^2)\psi\in L^2(\R,\C)\}.$$
The operator $H$ has compact resolvent and its spectrum is composed by the simple eigenvalues $(\lambda_k)_{k\in\N}$ such that 
$$\lambda_k=2 k +1.$$
A corresponding family of eigenfunctions, forming a Hilbert basis for $L^2(\R,\C)$, is given by the Hermite functions
$$\phi_k(x)= (-1)^{k}(\sqrt{\pi} 2^k k!)^{-1/2} e^{x^2/2} \big(e^{-x^2}\big)^{(k)},\qquad \forall k\in\N.$$
Each eigenfunction $\phi_k$ can also be defined in terms on 
Recall that 
Hermite polynomial $H_k(x) $ can be defined as follows
$$H_k(x)= (-1)^{k} e^{-x^2} \big(e^{-x^2}\big)^{(k)},\qquad \forall k\in\N$$
and hence
$$\phi_k(x)= (\sqrt{\pi} 2^k k!)^{-1/2} e^{-x^2/2} H_k(x),\qquad \forall k\in\N.$$

\begin{remark}
A Hilb's type formula applied to the Hermite polynomial (see \cite[Theorem 8.22.6 and equation (8.22.8)]{Szego}) yields, for every $m\in\N$,
\begin{align}
\frac{(m+1)!}{(2m + 1)! } e^{-x^2/2} H_{2m} (x) &= \cos(\sqrt{4m+1}x - m\pi) \nonumber\\
+\frac{x^3}{6\sqrt{4m+1}}&\sin(\sqrt{4m+1}x - m\pi) + O \left( \frac{1}{ m } \right),
\label{hib_1}
\end{align}
\begin{align}
 \frac{(m+2)! \sqrt{4m+3}}{(2m + 3)! } e^{-x^2/2} H_{2m+1} (x) &= \cos\Big(\sqrt{4m+3}x - (2m+1)\frac{\pi}{2}\Big)\nonumber\\
 &+\frac{x^3}{6\sqrt{4m+3}}\sin\Big(\sqrt{4m+3}x - (2m+1)\frac{\pi}{2}\Big)\nonumber\\
 &+ O \left( \frac{1}{ m } \right).
 \label{hib_2}
\end{align}
Now, we observe that
$$\lim_{m\rightarrow + \infty} (2m)^{-\frac{1}{4}}\frac{(m+1)!}{(2m + 1)! } 2^m \sqrt{\sqrt{\pi} (2m)!}=\frac{\sqrt{\pi}}{2^{\frac{5}{4}}},$$
$$ \lim_{m\rightarrow + \infty} (2m+1)^{-\frac{1}{4}}\frac{(m+2)! \sqrt{4m+3}}{(2m + 3)! } \sqrt{\sqrt{\pi} 2^{2m+1} (2m+1)!}=\frac{\sqrt{\pi}}{2^{\frac{5}{4}}}.$$
The last identities and~\eqref{hib_1}-\eqref{hib_2} yield that
\begin{align*}
\phi_{k} (x) &= c_k\left( \cos\Big(\sqrt{2k+1}x - k\frac{\pi}{2}\Big)\right.\\
&\left.+\frac{x^3}{6\sqrt{2k+1}}\sin\Big(\sqrt{2k+1}x - k\frac{\pi}{2}\Big) + O \left( \frac{1}{ k } \right)\right),
\end{align*}
where $c_k$ are some constants, depending on $k$, such that
$$\lim_{k\rightarrow+\infty} {c_k}{ k^{\frac{1}{4}}}=\frac{2^{\frac{5}{4}}}{\sqrt{\pi}}. $$
Hence, there exists $C_1>0$ such that, for every $x\in\R$, we have
\begin{align}
\label{bound_h}
 |\phi_k(x)|\leq C_1k^{-\frac{1}{4}},\quad\quad\forall k\in\N.
\end{align}
\end{remark}

\subsection{Well-posedness}

To ensure the well-posedness of~\eqref{0.1_h}, we consider as in the previous frameworks the following bilinear Schrödinger equation in the presence of a source term:
\begin{equation}
\label{0.1_h_bis}
\begin{cases}
 i \p_t \psi =\big[-\Delta+x^2\big]\psi+ u(t)\mu(x)\psi + f(t,x) & \quad x\in \R,\ t>0,\\
 \psi_0(t=0)=\psi_0\in L^2(\R,\C).
\end{cases}
\end{equation}
Let $A:=\{a_j\}_{j\leq N}\subset \R$ with $N\in\N^*$. For $n\in\N^*$, we introduce the space
\begin{equation*}
\widetilde H^1(A)=\left\{\psi \in L^2(\R,\C)\,\mid\, x^n\psi \in L^2,\, \psi\, \text{is $H^1$ up to the discontinuities}\, \{a_j\}_{j\leq N}\right\},
\end{equation*}
endowed with the norm
$$\|\cdot\|_{\widetilde H^1(A)}=\sqrt{\|\cdot\|_{H^1(-\infty,a_1)}^2+...+\|\cdot\|_{H^1(a_N,+\infty)}^2+\big\|(1+|x|)\cdot\big\|_{L^2(\R)}^2}.$$
The following proposition provides the well-posedness of~\eqref{0.1_h_bis} (and~\eqref{0.1_h}) in the space
$$H^1_h=\big\{\psi\in L^2(\R,\C)\,:\, (-\p_x+ x)\psi\in L^2(\R,\C)\big\},$$
endowed with the norm $$\|\cdot\|_{H^1_h}=\sqrt{\sum_{k\in\N}\lambda_k\big|\lag\cdot,\phi_k\rag_{L^2}\big|^2}=\sqrt{\sum_{k\in\N}(2k+1)\big|\lag\cdot,\phi_k\rag_{L^2}\big|^2}.$$

\begin{proposition}\label{well_h}
Let
$N \in \N^*$,
$A = \{a_j\}_{j \leq N} \subset (0,1)$ and
$T>0$.
Let $\mu\in \widetilde H^1(A)$, $\psi_0\in H^{1}_h$ and $f \in L^2( (0,T) , \widetilde H^{1}(A))$.
For any $u\in L^2( (0,T) ,\R)$, there exists a unique mild solution $\psi \in C( [0,T] ,H^{1}_h)$ to the problem~\eqref{0.1_h_bis} such that $\psi(0)=\psi_0$ that is a solution to
$$
\psi(t)=e^{-iH t} \psi_0-i\int_0^te^{-iH(t-s)}(u(s)\mu \psi(s) +  f(s))  d s, \quad t\in [0,T].
$$
Moreover, there is a constant~$C_T>0$ such that
$$\|\psi\|_{C( [0,T], H^{1}_h )}\le C_T \left (\|\psi_0\|_{H^1_h}+\|f\|_{L^2( (0,T) , \widetilde H^{1}(A))}\right).$$
\end{proposition}
\begin{proof}
The proof is again very similar to that of Theorem~\ref{well_d}: Recall that $D(|H|^\frac{1}{2}) = H_h^1$. The central part consists in estimating $\sqrt{ 2k+1}\lag \phi_k,f\rag_{L^2}$ for every $k\in\N$ when $f \in L^2((0,T), \widetilde{H}^{1}(A))$. The properties of the Hermite functions imply
\begin{align*}
\sqrt{2k}\lag \phi_k,f(s)\rag_{L^2}&=\lag (x-\p_x)\phi_{k-1},f(s)\rag_{L^2}=\lag\phi_{k-1}, \big(xf(s)+f'(s)\big)\rag_{L^2}\\
&+\sum_{l=1}^N\phi_{k-1}(a_l)[f(s,a_l^+)-f(s,a_l^-)],
\end{align*}
for every $k\in\N^*$. Thanks to~\eqref{bound_h}, the proof follows as for Theorem~\ref{well_d}.
\end{proof}

\subsection{Controllability of the linearised system}
One may wonder if the local exact controllability holds in the space $H^1_h$, for instance around the ground state $\phi_0$. The result seems provable if there exists $C > 0$ satisfying
\begin{equation}\label{badbad}
|\lag \mu\phi_0, \phi_{k}\rag_{L^2}| \ge \frac{C}{\sqrt{\lambda_k}}=\frac{C}{\sqrt{2k+1}},\quad  \forall k\in\N.
\end{equation}
However,~\eqref{badbad} cannot be verified when $\mu \in L^2(\R,\C)$ since $(\lag \mu\phi_0, \phi_{k}\rag_{L^2})_{k\in\N}\in\ell^2$. From this perspective, it seems to be necessary to establish the well-posedness of~\eqref{0.1_h} in a higher regularity space via a stronger regularising effect than the one used in Proposition~\ref{well_h}.

In order to identify a suitable space for this purpose, we can study a linearised equation of~\eqref{0.1_h} and investigate the space where its global exact controllability can be ensured.
We consider thus the linearised equation of~\eqref{0.1_h} with respect to the eigensolution $\phi_0(t)=e^{-i\lambda_0 t}\phi_0$ corresponding to the control $u=0$. The linear system reads
\begin{equation}\label{lin_h}\begin{cases}
 i \p_t \xi=H \xi + v(t)\mu (x) \phi_0(t), & \quad x\in \R,\ t>0,\\
 \xi(0,x)=0 & \quad x\in \R.
\end{cases}
\end{equation}
For $a\in\R$, we introduce the space
$$H_{h,a}=\{\psi\in L^2\,:\, \|\psi\|_{H_{h,a}}<\infty\},$$
where
$$\|\cdot\|_{H_{h,a}}=\sqrt{
|\lag \cdot,\phi_0\rag_{L^2}|^2
+
\sum_{k\in\N*}{k}{\phi_{k-1}^{-2}(a)}|\lag \cdot,\phi_k\rag_{L^2}|^2
}.$$
We also consider the space
$$X_t=\{\psi\in L^2(\R,\C)\,:\, \Re\big(\lag\psi,\phi_0(t) \rag_{L^2}\big)=0\},\quad\quad t>0.$$ 
We show that, when $\mu= 1_{[a,+\infty)}$, a suitable space where the global exact controllability of~\eqref{lin_h} can be ensured is
$$H_{h,a}\cap X_T.$$
Note that the asymptotic behaviour of Hermite's functions~\eqref{bound_h} infers
$$H^1(\R,\C)\supset D(|H|^\frac{3}{4})\supseteq H_{h,a}.$$
From this perspective, if we want to apply the Inverse Mapping Theorem as in the previous frameworks (Dirichlet or Periodic case), a stronger well-posedness result than Proposition~\ref{well_h} is required.

\begin{theorem}\label{theoremh}
Let $a\in\R$ and $\mu=1_{[a,+\infty)}$. We have the two following properties.

\smallskip

\noindent
{\bf 1.} Let $T>0$ and $u\in L^2( (0,T) ,\R)$. There exists a unique mild solution $\xi \in C( [0,T] ,H_{h,a})$ to the problem~\eqref{lin_h} such that $
 \xi(t)\in H_{h,a}\cap X_t$ for every $t\in [0,T]$ and defined by the Duhamel's formula:
$$
 \xi(t)=-i\int_0^te^{-iH(t-s)} u(s)\mu\phi_0(s)\,\mathrm{d}s\in H_{h,a}\cap X_t, \quad t\in [0,T] .$$

\smallskip

\noindent
{\bf 2.} For any $T>\pi$ and $\psi_1 \in H_{h,a}\cap X_T$, 
 there exists a control $u\in L^2((0,T),\R)$ such that the corresponding mild solution $\xi$ to the problem~\eqref{lin_h} with $\xi(0)=0$ satisfies $$\xi(T)= \psi_1.$$
 \end{theorem}
\begin{proof}{\bf 1.} The first point is a consequence of the following identity with $k\in\N^*$:
\begin{align}
 \lag \phi_{k}, \mu\phi_0 (t)\rag_{L^2(\R)} &= \lag \phi_{k}, \phi_0(t) \rag_{L^2[a,+\infty)} = \frac{(-1)^ke^{-i\lambda_0 t}}{\sqrt{\pi 2^{k+1} k!}} \lag (e^{-x^2/2})^{(k)}, H_0 \rag_{L^2[a,+\infty)} \nonumber\\
 &= -\frac{(-1)^ke^{-i t}}{\sqrt{\pi 2^{k+1} k!}} \big[ (e^{-x^2/2})^{(k-1)} \big](x=a^+) = \frac{e^{-i  t}}{2 \sqrt{k \sqrt{\pi}}} \phi_{k-1}(a).
 \label{ee}
\end{align}
The last identity and Corollary~\ref{well-ingham} yields the existence of $C>0$ uniform in $[0,T]$ such that
\begin{align*}
\|G(t)\|_{ H_{h,a} }^2&= \Big| \int_0^t  \|\phi_0\|^2_{L^2[a,+\infty)}\,\mathrm{d}s\Big|^2+\sum_{k\in\N^*}\frac{k}{\phi_{k-1}^2(a)}\Big| \int_0^t e^{i\lambda_ks } \lag \phi_k,\mu\phi_0(s)\rag_{L^2}\,\mathrm{d}s\Big|^2  \\
&= t^2\|\phi_0\|^4_{L^2[a,+\infty)}+\frac{1}{4 \sqrt{\pi}}\sum_{k\in\N^*}\Big|\int_0^t e^{i(\lambda_k-1)s }\,\mathrm{d}s\Big|^2<C.
\end{align*}
Now, $\xi(t)\in H_{h,a}$ for every $t\in [0,T]$ and $\xi(t)\in X_T$ since $u$ and $\mu$ are real valued. Finally, the last relation concludes the proof of the first point of the proposition as in Lemma~\ref{wellposedness-bound-3.d}.

\smallskip

\noindent
{\bf 2.} The second point follows by the approach leading to Proposition~\ref{lin-exact-d}. The controllability can be ensured by proving the solvability of a moment problem:
\begin{align*}
\int_0^T e^{i (\lambda_{k}-\lambda_0) s} v(s) ds &= \frac{i e^{i\lambda_{k} T}\lag \xi(T),\phi_{k}\rag_{L^2}}{\lag \mu\,\phi_0 ,\phi_k\rag_{L^2}}
=:x_k,\qquad  k\in\N. 
\end{align*}
Note that, thanks to~\eqref{ee}, we have $$x_k=-i 2 \pi^\frac{1}{4} \sqrt{k } \phi_{k-1}(a) \lag \xi(T),\phi_{k}(T)\rag_{L^2}$$
and then $(x_k)_{k\in\N}\in \ell^2_r(\N,\C)=\left\{(x_k)_{k\in\N}\in \ell^2: x_0\in\R\right\}$ since $\xi(T)\in H_{h,a}\cap X_T$. Finally, the solvability of the moment problem is implied for $T>\pi$ by Proposition~\ref{solvability_moment_problem}.\end{proof}

\newpage

\appendix

\section{Proof of Propositions~\ref{smooth_d} and~\ref{smooth_p}}\label{endpoint}

\begin{proof}[Proof of Proposition~\ref{smooth_d}]
Let us consider $l=1$. The proof is the same in the general case. Thanks to Theorem~\ref{well_d}, the map $u\in L^2((0,T),\R)
\mapsto \Psi_T^1\in H^1_0\cap \TT_{\phi_1(T)}$ is continuous.

\smallskip
\noindent
{\bf 1) Differentiability.} We introduce the following elements.
\begin{itemize}
 \item Let $\psi$ be the mild solution
to~\eqref{0.1_d} with $u\in L^2((0,T),\R)$ and $\psi_0=\phi_1$:
 \[
  \psi(t)=e^{i\Delta t} \phi_1-i\int_0^te^{i\Delta(t-s)} u(s)\mu \psi(s)\,\mathrm{d} s.
 \]
 \item Let $\xi$ be the corresponding linearised mild solution 
 with $v\in L^2((0,T),\R)$:
 \[
   \xi(t)=-i \int_0^t e^{i \Delta (t-s)}\left(u(s)\mu \xi(s)+ v(s)\mu \psi (s)\right)\, \mathrm{d}s.
 \]
 \item Let $\widetilde \psi$ be the mild solution to~\eqref{0.1_d} with control $u+v$ and $\psi_0=\phi_1$:
 \[
  \widetilde{\psi}(t)=e^{i\Delta t} \phi_1-i\int_0^te^{i\Delta(t-s)} (u(s)+v(s))\mu \widetilde{\psi}(s)\,\mathrm{d} s.
 \]
\end{itemize}
Let $\varphi = \widetilde \psi - \psi -\xi$, then
\[
 \varphi(t)=-i\int_0^te^{i\Delta(t-s)} (u(s)+v(s))\mu \varphi(s)\,\mathrm{d} s
 -i\int_0^te^{i\Delta(t-s)} v(s) \mu \xi(s)\,\mathrm{d} s.
\]
Let $R>0$ such that $\|u\|_{L^2},\|u+v\|_{L^2}< R$. Thanks to Theorem~\ref{well_d}, there exist $C_1,C_2>0$ uniformly bounded in $[0,T]$ such that
\begin{align*}
 \|\varphi\|_{C ([0,T],H^1_0)}& \leq C_1 \| v \mu \xi \|_{L^2((0,T), H^1(A))}\leq C_2 \| v\|_{L^2} \|\xi \|_{C ([0,T],H^1_0)}\\
 &\leq C_2 C_1 \| v\|_{L^2} \| v \mu \psi \|_{L^2((0,T), H^1(A))}
\leq C_2^2 \| v\|_{L^2}^2 \| \psi \|_{C ([0,T],H^1_0)}\\
&\leq C_2^2C_1 \| v\|_{L^2}^2 \| \phi_1 \|_{H^1_0} = \pi C_2^2C_1 \| v\|_{L^2}^2.
\end{align*}
The last relation proves that $\|\varphi\|_{C ([0,T],H^1_0)}=O(\|v\|_{L^2}^2)$ and ensure the differentiability of $\Psi_T^1$.

\smallskip

\noindent
{\bf 2) Continuity of the differential.} We introduce the following elements for $u^1,u^2,v\in L^2((0,T),\R).$
\begin{itemize}
 \item Let $\psi^1$ be the mild solution to~\eqref{0.1_d} with control $u^1$ and $\psi_0=\phi_1$:
 \[
  \psi^1(t)=e^{i\Delta t} \phi_1-i\int_0^te^{i\Delta(t-s)} u^1(s)\mu \psi^1(s)\,\mathrm{d} s.
 \]
 \item Let $\xi^1$ be the corresponding linearised mild solution
 with control $v$:
 \[
   \xi^1(t)=- i \int_0^t e^{i \Delta (t-s)}\left(u^1(s)\mu \xi^1(s)+ v(s)\mu \psi^1 (s)\right)\, \mathrm{d}s.
 \]
 \item Let $\psi^2$ be the mild solution to~\eqref{0.1_d} with control $u^2$ and $\psi_0=\phi_1$:
 \[
  \psi^2(t)=e^{i\Delta t} \phi_1-i\int_0^te^{i\Delta(t-s)} u^2(s)\mu \psi^2(s)\,\mathrm{d} s.
 \]
 \item Let $\xi^2$ be the corresponding linearised mild solution
 with control $v$:
  \[
   \xi^2(t)=-i \int_0^t e^{i \Delta (t-s)}\left(u^2(s)\mu \xi^2(s)+ v(s)\mu \psi^2 (s)\right)\, \mathrm{d}s.
 \]

\end{itemize}
We have $\big[\partial_{u} \Psi_T^1( u^1)-\partial_{u} \Psi_T^1( u^2)\big] \cdot v=P_T^1\big(\xi^1(T) - \xi^2(T) \big) =P_T^1(\theta)$ where
\begin{align*}
\theta(t)&=-i \int_0^t e^{i \Delta (t-s)} u^1(s)\mu \theta(s)\, \mathrm{d}s
-i \int_0^te^{i \Delta (t-s)} (u^1(s)-u^2(s))\mu \xi^2(s)\, \mathrm{d}s\\
&\quad-i \int_0^te^{i \Delta (t-s)} v(s)\mu \left(\psi^1(s)-\psi^2(s)\right)\, \mathrm{d}s.
\end{align*}
Let $R>0$ such that $\|u^1\|_{L^2},\|u^2\|_{L^2}< R$. Note now that
\begin{align*}
\psi^1(t)-\psi^2(t)&=-i \int_0^te^{i \Delta (t-s)} u^1(s)\mu (\psi^1(s)-\psi^2(s))\,\mathrm{d}s\\
 &\qquad-i \int_0^te^{i \Delta (t-s)} (u^2(s)-u^1(s))\mu \psi^2(s)\,\mathrm{d}s
\end{align*}
and hence
Theorem~\ref{well_d} yields the existence of $C_1>0$ uniformly bounded in $[0,T]$, depending on $R$, such that
\begin{align}\label{bound_d}\|\psi^1-\psi^2\|_{C ([0,T],H^1_0)}\leq C_1 \|(u^1-u^2)\mu \psi^2\|_{L^2((0,T), H^1(A))}.\end{align}
From~\eqref{bound_d} and Theorem~\ref{well_d}, there exists $C_2>0$ uniformly bounded in $[0,T]$, depending on $R$, such that
\begin{align*}
 \|\theta\|_{C ([0,T],H^1_0)}& \leq C_1 \big\| \big(u^1(t)-u^2(t)\big)\mu(x) \xi^2+ v(t)\mu (x) \big(\psi^1-\psi^2\big) \big\|_{L^2((0,T), H^1(A))}\\
 &\leq C_2 \Big( \|u^1-u^2 \|_{L^2} \|\xi^2 \|_{C ([0,T],H^1_0)}+ \|v \|_{L^2} \|\psi^1-\psi^2\|_{C ([0,T],H^1_0)}\Big) \\
 &\leq C_2^2 \Big( \|u^1-u^2 \|_{L^2}\|v \|_{L^2} \|\psi^2 \|_{C ([0,T],H^1_0)}\\
 &\qquad+\|v \|_{L^2} \|u^1-u^2 \|_{L^2}\|\psi^2\|_{C ([0,T],H^1_0)}\Big) \\
 &\leq 2 C_1 C_2^2 \|u^1-u^2 \|_{L^2}\|v \|_{L^2} \|\phi_1 \|_{H^1_0}\leq 2\pi R C_1 C_2^2 \|u^1-u^2 \|_{L^2}.
\end{align*}
The last relation infers that $\partial_{u} \Psi_T^1( u)$ is actually locally Lipschitz in $u$ ensuring the continuity of the differential. The genral case of $l\in\N^*$ is proved in the very same way.
\end{proof}

\begin{proof}[Proof of Proposition~\ref{smooth_p}]
The proof is the same as that of Proposition~\ref{smooth_d}, with the difference that we need to use Theorem~\ref{well_p} instead of Theorem~\ref{well_d}. In this case, the operator $H$ replaces the Dirichlet Laplacian and $H^1_p$  the space $H^1_0$.
\end{proof}

\section{Solvability of moment problems}
\label{mome_appendix}
In this appendix, we present two results on the solvability of moment problems, which are used in various parts of this work. Very similar statements can be found in \cite[Appendix B.2]{BL-2010} or \cite[Appendix A]{Duc20b} but they do not directly apply to our frameworks. For the sake of completeness, we provide their proofs.

\begin{proposition}\label{solvability_moment_problem}
Let $(\lambda_k)_{k\in\Z}$ be a sequence of real numbers such that
\[
\lambda_0=0,\qquad \lambda_k\neq \pm \lambda_j,\quad\forall k\in\Z,\,\forall j\in\Z,\,k\neq j.
\]
Let $(\mu_k)_{k\in\Z}$ be the increasing sequence such that  
$$\big\{\mu_k,\, k\in \Z\big\}=\big\{\lambda_0,\,\pm \lambda_k\mid \,k\in \N^*\big\}.$$
Assume $\inf_{k\neq
j}|\mu_{k}-\mu_{j}|>0$. Let
$$\gamma:=\sup_{\substack{K\subset\Z\\K \text{ finite }}}\inf_{\underset{k\neq j}{k,j\in\Z\setminus K}}|\mu_{k}-\mu_{j}|.$$ 
Fix $T>2\pi/\gamma$, for every $(x_k)_{k\in\Z}\in \ell^2(\C)$ such that $x_0\in\R$, there exists $u\in L^2((0,T),\R)$ such that
\begin{equation*}
x_k=\int_0^Tu(s)e^{-i\lambda_k s}ds,\qquad \forall k\in
\Z.\end{equation*}
\end{proposition}
\begin{proof}
Let $T>2\pi/\gamma$. Thanks to \cite[Theorem 4.6]{komornik}, the functions $\{e^{i\mu_k t}\}_{k\in\Z}$ form a Riesz Basis of the space
$X:=\overline{span\{e^{i\mu_k t}:k\in\Z\}}^{ L^2}\subseteq L^2((0,T),\C).$ Hence, the map $$M:g\in X\longmapsto \big(\lag
g,e^{i\mu_k t}\rag_{L^2}\big)_{k\in\Z}\in \ell^2$$ is invertible.
Let $(x_k)_{k\in\Z}\in\ell^2$. We define $(y_k)_{k\in\Z}\in\ell^2$ as follows.

\begin{itemize}
 \item For every $k\in\Z$ such that $\mu_k=\lambda_j$ with $j\in\Z^*$, we denote $y_k=x_j$.
 \item For every $k\in\Z$ such that $\mu_k=-\lambda_j$ with $j\in\Z^*$, we denote $y_k=\overline{x_j}$.
 \item For $k\in\Z$ such that $\mu_k=\lambda_0$, we denote $y_k=x_0$.
\end{itemize}
Now, the invertibility of the map $M$ yields the existence of $u\in
X\subseteq L^2((0,T),\C)$
such that
\begin{equation*}{y_{k}}=\int_{0}^Tu(t)e^{-i\mu_k t}dt,\qquad  \forall
k\in\Z.\end{equation*}
The last relation infers, for every $k\in\Z$,
\begin{equation}\label{pie}
{x_{k}}=\int_{0}^Tu(t)e^{-i\lambda_k t}dt,\quad
{\overline{x_{k}}}=\int_{0}^Tu(t)e^{i\lambda_k t}dt, \quad
x_0=\int_{0}^Tu(t)dt.
\end{equation}
The relations~\eqref{pie} and $x_0\in\R$ imply, for every $k\in\Z$,
\begin{equation}\label{pie_1}
\overline{x_{k}}=\int_{0}^T\overline{u}(t)e^{i\lambda_k t}dt,\quad
{{x_{k}}}=\int_{0}^T\overline{u}(t)e^{-i\lambda_k t}dt,\quad
0=\int_{0}^T \Im(u)(t)dt.
\end{equation}
Thanks to~\eqref{pie}-\eqref{pie_1}, we have, for every $k\in\Z$,
\begin{equation*}
0=\int_{0}^T \Im(u)(t)e^{i\lambda_k t}dt=\int_{0}^T \Im(u)(t)e^{-i\lambda_k t}dt=\int_{0}^T \Im(u)(t)dt.
\end{equation*}
The last relations imply $\lag \Im(u), e^{i\mu_k}\rag_{L^2}=0$ for every $k\in\Z.$ Finally, by definition of $(\mu_k)_{k\in\Z}$, since $X$ is invariant by complex conjugation, $\overline{u}\in X$ and $\Im(u)\in X$.
Finally, $u\in L^2((0,T),\R)$ and the first relation of~\eqref{pie} ensures the result.
\end{proof}

\begin{corollary}\label{well-ingham}
Let $(\lambda_k)_{k\in\Z}$ be a sequence of real numbers such that $\inf_{k\neq
j}|\lambda_{k}-\lambda_{j}|>0$.
For every $T>0$, there exist $C(T)>0$ such that, for every $u\in L^2((0,T),\C)$,
$$\Big\|\int_{0}^T u(t) e^{i\lambda_k t}dt\Big\|_{\ell^2}\leq C\|u\|_{L^2(0,T)}.$$
\end{corollary}
\begin{proof}
Let $T>2\pi/\gamma$ with $\gamma:=\sup_{K\subset\Z}\inf_{\underset{k\neq j}{k,j\in\Z\setminus K}}|\lambda_{k}-\lambda_{j}|,$ where $K$ runs over the finite subsets of $\Z$. Thanks to \cite[Theorem 4.6]{komornik}, the functions $\{e^{-i\lambda_k t}\}_{k\in\Z}$ form a Riesz Basis of the space
$X:=\overline{span\{e^{-i\lambda_k t}:k\in\Z\}}^{ L^2}.$
Now, there exists $C(T)>0$ such that $\sum_{k\in \Z}|\lag g,e^{-i\lambda_k (\cdot)}\rag_{L^2}|^2\leq
C(T)^2\|g\|_{L^2}^2$ for every $g\in X$. Let $P:L^2\longrightarrow X$ be
the orthogonal projector. For $u\in L^2((0,T),\C)$, we have
\begin{equation*}\big\|(\lag u, e^{-i\lambda_k
(\cdot)}\rag_{L^2})_{k\in\Z}\big\|_{\ell^2}=\big\|(\lag Pu, e^{-i\lambda_k
(\cdot)}\rag_{L^2})_{k\in\Z}\big\|_{\ell^2}\leq C(T)\|Pu\|_{L^2}\leq
C(T)\|u\|_{L^2}.
\end{equation*}
When $T\leq 2\pi/\gamma$, for $u\in L^2((0,T),\C)$, we define $\tilde u\in L^2((0,2\pi/\gamma+1),\C)$ such that $\tilde u = u$ on $(0,T)$ and $\tilde u=0$ in $(T, 2\pi/\gamma+1)$. Then
\begin{equation*}\Big\|\int_0^T e^{i\lambda_k \tau}{
u}(\tau)dt\Big\|_{\ell^2}=\Big\|\int_0^{2\pi/\gamma+1} e^{i\lambda_k \tau}{\tilde
u}(\tau)dt\Big\|_{\ell^2}\leq C(2\pi/\gamma+1)\|u\|_{L^2}.\qedhere
\end{equation*}
\end{proof}

\bibliographystyle{plain}
\bibliography{biblio}
\end{document}